\RequirePackage{silence}
\documentclass[11pt, english]{article}

\usepackage[margin= 1.6 cm]{geometry}

\usepackage{amsthm}
\usepackage{amsmath}
\usepackage{amssymb}
\usepackage{setspace}
\usepackage{mathtools}
\usepackage{verbatim}
\usepackage{csquotes}
\usepackage{graphicx}
\usepackage[hyphens]{url}
\usepackage[hidelinks]{hyperref}
\usepackage[hyphenbreaks]{breakurl}
\usepackage{thm-restate}
\usepackage{cleveref}
\usepackage{graphicx}
\usepackage[inline]{enumitem}
\usepackage{framed}
\usepackage{subcaption}
\usepackage{paralist}

\usepackage{caption}

\newcommand{\A}{\mathcal{A}}
\newcommand{\B}{\mathcal{B}}
\newcommand{\nn}{||.||}
\newcommand{\R}{\mathbb{R}}
\usepackage{floatrow}
\usepackage[T1]{fontenc}

\usepackage{tikz}
\usepackage{mathdots}
\usepackage{xcolor}
\usepackage{diagbox}
\usepackage{colortbl}
\usepackage[absolute,overlay]{textpos}

\graphicspath{ {./images/} }

\usetikzlibrary{calc}
\usetikzlibrary{decorations.pathreplacing}
\usetikzlibrary{positioning,patterns}
\usetikzlibrary{arrows,shapes,positioning}
\usetikzlibrary{decorations.markings}

\tikzstyle{edge}=[very thick]
\definecolor{bostonuniversityred}{rgb}{0.8, 0.0, 0.0}
\definecolor{arsenic}{rgb}{0.23, 0.27, 0.29}
\tikzstyle{diredge}=[postaction={decorate,decoration={markings,
		mark=at position .95 with {\arrow[scale = 1]{stealth};}}}]

\newcommand{\defPt}[3]{
	\def \pt {(#1, #2)}
	\coordinate [at = \pt, name = #3];
}

\tikzset{
   conn/.pic={
     \defPt{0.2}{-0.5}{q0}
     \defPt{-1}{-1.5}{q5}
    \defPt{1}{1.2}{q1}
    \defPt{1}{2.7}{q6}
    \defPt{1.25}{-1.2}{q2}
    \defPt{2.5}{0.6}{q3}
    \defPt{2.5}{-0.6}{q4}
  
        \draw[line width=1 pt] (q0) -- (q1) -- (q3) -- (q4);
        \draw[line width=1 pt] (q2) -- (q3);
        \draw[line width=1 pt] (q0) -- (q5);
        \draw[line width=1 pt] (q1) -- (q6);
  }
}

\newcommand{\fitellipsis}[3] 
{\draw []let \p1=(#1), \p2=(#2), \n1={atan2(\y2-\y1,\x2-\x1)}, \n2={veclen(\y2-\y1,\x2-\x1)}
    in ($ (\p1)!0.5!(\p2) $) ellipse [ x radius=\n2/2+0.3cm+#3cm, y radius=#3cm, rotate=\n1];
}

\newcommand{\fitellipsiss}[3] 
{\draw [fill=white]let \p1=(#1), \p2=(#2), \n1={atan2(\y2-\y1,\x2-\x1)}, \n2={veclen(\y2-\y1,\x2-\x1)}
    in ($ (\p1)!0.5!(\p2) $) ellipse [ x radius=\n2/2+#3cm, y radius=#3cm, rotate=\n1];
}

\newcommand{\fitellipsisss}[3] 
{\draw []let \p1=(#1), \p2=(#2), \n1={atan2(\y2-\y1,\x2-\x1)}, \n2={veclen(\y2-\y1,\x2-\x1)}
    in ($ (\p1)!0.5!(\p2) $) ellipse [ x radius=\n2/2+#3cm, y radius=#3cm, rotate=\n1];
}

\renewcommand{\Bbb}{\mathbb}
\theoremstyle{plain}

\newtheorem*{thm*}{Theorem}
\newtheorem{thm}{Theorem}[section]
\Crefname{thm}{Theorem}{Theorems}

\newtheorem*{lem*}{Lemma}
\newtheorem{lem}[thm]{Lemma}
\Crefname{lem}{Lemma}{Lemmas}

\newcounter{claimcount}
\newenvironment{claim}{\refstepcounter{claimcount}\textbf{Claim \arabic{claimcount}:}}{}

\usepackage{etoolbox}
\AtBeginEnvironment{proof}{\setcounter{claimcount}{0}}

\Crefname{claimcount}{Claim}{Claims}

\newtheorem{prop}[thm]{Proposition}
\Crefname{prop}{Proposition}{Propositions}

\Crefname{cor}{Corollary}{Corollaries}

\Crefname{conj}{Conjecture}{Conjectures}

\Crefname{qn}{Question}{Questions}

\Crefname{obs}{Observation}{Observations}

\newtheorem{ex}[thm]{Example}
\Crefname{ex}{Example}{Examples}

\theoremstyle{definition}

\Crefname{prob}{Problem}{Problems}

\Crefname{defn}{Definition}{Definitions}

\newtheorem{fact}[thm]{Fact}
\Crefname{fact}{Fact}{Facts}

\newtheorem*{defn*}{Definition}

\theoremstyle{remark}
\newtheorem*{rem}{Remark}

\renewenvironment{proof}[1][]{\begin{trivlist}
\item[\hspace{\labelsep}{\bf\noindent Proof#1.\/}] }{\qed\end{trivlist}}

\newenvironment{cla_proof}[1][]{\begin{trivlist}
\item[\hspace{\labelsep}{\noindent \emph{Proof#1.}\/}] }{\qed\end{trivlist}}

\newcommand{\ceil}[1]{
    \left\lceil #1 \right\rceil
}
\newcommand{\floor}[1]{
    \left\lfloor #1 \right\rfloor
}
\newcommand{\midd}{:}

\newcommand{\eps}{\varepsilon}

\newcommand{\su}{\subseteq}
\DeclareMathOperator{\dist}{dist}
\DeclareMathOperator{\spn}{span}

\expandafter\def\expandafter\normalsize\expandafter{%
    \normalsize
    \setlength\abovedisplayskip{8pt}
    \setlength\belowdisplayskip{8pt}
    \setlength\abovedisplayshortskip{4pt}
    \setlength\belowdisplayshortskip{4pt}
}

\usepackage[square,sort,comma,numbers]{natbib}
 \setlist[itemize]{leftmargin=*}

\newcommand{\diam}{{\rm{diam}}}

\DeclareFontFamily{OT1}{pzc}{}
\DeclareFontShape{OT1}{pzc}{m}{it}{<-> s * [1.10] pzcmi7t}{}
\DeclareMathAlphabet{\mathpzc}{OT1}{pzc}{m}{it}

\title{\vspace{-0.8cm} Unit and distinct distances in typical norms}

\author{Noga Alon\thanks{Department of Mathematics, Princeton
University, Princeton, USA. Research supported in part by NSF grant
DMS-2154082 and by USA-Israel BSF grant 2018267. Email: \href{mailto:nalon@math.princeton.edu} {\nolinkurl{nalon@math.princeton.edu}}.} \and
Matija Buci\'c\thanks{Department of Mathematics, Princeton University, Princeton, USA. Research supported in part by NSF grants CCF-1900460 and DMS-2349013. Email: \href{mailto:matija.bucic@ias.edu} {\nolinkurl{matija.bucic@ias.edu}}.}
\and
Lisa Sauermann\thanks{Institute for Applied Mathematics, University of Bonn, Bonn, Germany. Part of this work was completed while this author was affiliated with Massachusetts Insitute of Technology. Research supported in part by NSF Award DMS-2100157 and a Sloan Research Fellowship. Email: 
\href{mailto:sauermann@iam.uni-bonn.de} {\nolinkurl{sauermann@iam.uni-bonn.de}}.}
}
 \date{}

\begin{document}

\maketitle

\vspace{-0.5cm}
\begin{abstract}
Erd\H{o}s' unit distance problem and Erd\H{o}s' distinct distances problem are among the most classical and well-known open problems in discrete mathematics. They ask for the maximum number of unit distances, or the minimum number of distinct distances, respectively, determined by $n$ points in the Euclidean plane. The question of what happens in these problems if one considers normed spaces other than the Euclidean plane has been raised in the 1980s by Ulam and Erd\H{o}s and attracted a lot of attention over the years. We give an essentially tight answer to both questions for almost all norms on $\R^d$, in a certain Baire categoric sense.

For the unit distance problem we prove that for almost all norms $\|.\|$ on $\R^d$, any set of $n$ points defines at most $\frac{1}{2} d \cdot n \log_2 n$ unit distances according to $\|.\|$. We also show that this is essentially tight, by proving that for \emph{every} norm $\|.\|$ on $\R^d$, for any large $n$, we can find $n$ points defining at least $\frac{1}{2}(d-1-o(1))\cdot n \log_2 n$ unit distances according to $\|.\|$.

For the distinct distances problem, we prove that for almost all norms $\|.\|$ on $\R^d$ any set of $n$ points defines at least $(1-o(1))n$ distinct distances according to $\|.\|$. This is clearly tight up to the $o(1)$ term.

We also answer the famous Hadwiger--Nelson problem for almost all norms on $\R^2$, showing that their unit distance graph has chromatic number $4$.

Our results settle, in a strong and somewhat surprising form, problems and conjectures of Brass, Matou\v{s}ek, Brass--Moser--Pach, Chilakamarri, and Robertson. The proofs combine combinatorial and geometric ideas with tools from Linear Algebra, Topology and Algebraic Geometry.
\end{abstract}

\section{Introduction}
\label{S1}

\subsection{Unit distances}

Erd\H{o}s' unit distance problem raised in 1946 in
\cite{Erdos-unit-distance} (see also \cite[Chapter 5]{BMP-survey})
is among the best-known open problems
in combinatorics. The problem asks about estimating 
the maximum possible number $U_{\|.\|_2}(n)$ of unit distances
determined by $n$ distinct points in the Euclidean plane $\R^2$ according to the Euclidean norm $\|.\|_2$.
The best bounds to date are 
\[n^{1+\Omega(1/\log \log n)} \le U_{\|.\|_2}(n)
\le O(n^{4/3}).\]
The lower bound appeared in the initial paper
of Erd\H{o}s \cite{Erdos-unit-distance}, and the upper bound was first
proved
by Spencer, Szemer\'{e}di and Trotter \cite{SST-ub} in 1984, see also
\cite{szekely-simple-proof} for a short and elegant 
proof based on the
Crossing Lemma. More on the rich history of this problem can be found
in the surveys \cite{BMP-survey,szemeredi-survey}.

An interesting variant of the problem deals with the same question
in general real normed spaces. This was first suggested by Ulam and
described explicitly by
Erd\H{o}s in the early 1980s \cite{erdos-problem}. 
We call a norm $\|.\|$ on $\R^d$ a $d$-norm and denote by $U_{\|.\|}(n)$ the maximum possible number of unit distances determined by a set of $n$ distinct points in $\R^d$. The unit distance graph of a $d$-norm $\|.\|$ is the graph whose vertices are all points of $\R^d$, where two points are adjacent if and only if the distance between them is one. Thus, $U_{\|.\|}(n)$ is the maximum number of edges in an $n$-vertex subgraph of the unit distance graph of $\|.\|$.
An initial, simple observation is that if the boundary of the unit ball of a $d$-norm $\|.\|$ contains a straight-line segment then $U_{\|.\|}(n)$ is quadratic in $n$, as in this case there are two infinite subsets $A,B$ of $\R^d$ so that the distance between any $\textbf{a} \in A$ and $\textbf{b} \in B$ is one.
On the other hand, if $\|.\|$ is a $2$-norm which is strictly convex  (meaning that the boundary of the unit ball contains no straight line segment), then one can extend the known proofs from the Euclidean case to show that $U_{\|.\|}(n)\le O(n^{4/3})$. 
Valtr \cite{Va} constructed a strictly convex $2$-norm in which for every $n$ there exist $n$-element point sets with at least $\Omega(n^{4/3})$ unit distances. See also \cite{solymosi2022arrangements} for a much more general family of norms with the same property. This shows that the upper bound cannot be improved in general. See \cite{BMP-survey} for more details about the history of the problem of estimating $U_{\|.\|}(n)$ for general norms.

It is not difficult to see that for any $2$-norm $\|.\|$, we have 
$U_{\|.\|}(n) \geq (\frac{1}{2}-o(1))n \log_2 n$. 
Indeed, the graph of the $k$-dimensional hypercube is a subgraph of the unit distance graph of any $2$-norm $\|.\|$ as shown by choosing $k$ random unit vectors $\textbf{u}_1,\textbf{u}_2, \ldots,\textbf{u}_k \in \R^2$ and defining $\textbf{v}_S=\sum_{i \in S} \textbf{u}_i$ for every subset $S \su \{1,2,\ldots,k\}$. If two subsets $S,S'\su \{1,2,\ldots,k\}$ differ in a single element, then the distance according to $\|.\|$ between $\textbf{v}_S$ and $\textbf{v}_{S'}$ is one, as desired.

A remarkable, often repeated (see e.g.\ \cite{gil_matousek,szemeredi-survey,Sw,sheffer_matousek,matouvsek2011dawn,BMP-survey}), result of Matou\v{s}ek \cite{matousek} from 2011  shows that for \emph{most} norms, in a Baire Categoric sense (that will be described in detail in Section \ref{S2}), this is not far from optimal. Namely for a typical $2$-norm $\|.\|$, $U_{\|.\|}(n) \leq O(n \log n \log \log n)$. This suggests the obvious problem of deciding whether or not the $\log \log n$ term is necessary. The question of how small $U_{\|.\|}(n)$ can be has already been considered by Brass \cite{Br0} in 1996, who in particular asked if there exists a $2$-norm with $U_{\|.\|}(n) \leq O(n \log n)$. The corresponding problem in higher dimensions has been considered as well. In particular, Brass, Moser and Pach \cite[Chapter 5, Problems 4 and 5, p.\ 195]{BMP-survey} conjectured that for every $d \geq 3$ and every $d$-norm $\|.\|$, $U_{\|.\|}(n)$ is asymptotically larger than $\Omega(n \log n)$ and asked whether or not for $d \geq 4$ there is a $d$-norm $\|.\|$ so that $U_{\|.\|}(n)= o(n^2)$. Note that for the $d$-dimensional Euclidean norm $\|.\|_2$ it is easy to see that $U_{\|.\|}(n) \geq \Omega(n^2)$ for every $d \geq 4$ (and in fact, the precise constant is known as well, see \cite{Er1}), showing that for the Euclidean norms, the problem is simple in all dimensions $d \geq 4$, despite being wide open in dimensions $2$ and $3$.

In the present paper, we settle the above-mentioned questions of 
Brass, of Matou\v{s}ek and of Brass, Moser and Pach in all dimensions in the following strong form.

\begin{thm}\label{thm:main} 
For any $d\ge 2$, for most $d$-norms $\|.\|$, we have
\[U_{\|.\|}(n) \leq \frac{d}{2}\cdot n \log_2 n\]
for all $n$. More precisely, for all $d$-norms $\|.\|$ besides a meagre set, the following holds: 
For every $n\ge 1$ and every set of $n$ points in $\R^d$, 
there are at most $\frac{d}{2}\cdot n \log_2 n$ unit distances according to $\|.\|$ among the $n$ points.
\end{thm}

\begin{thm}
\label{t12}
Let $d\ge 2$ be fixed. Then for every $d$-norm $\|.\|$, we have \[U_{\|.\|}(n) \geq \frac{d-1-o(1)}{2} \cdot n \log_2 n.\] for all large $n$. That is, for every $d$-norm $\|.\|$ and every $n$, there exists a set of $n$ points in $\R^d$ such that the number of unit distances according to $\|.\|$ among the $n$ points is at least $\frac{d-1-o(1)}{2} \cdot n \log_2 n$, where the $o(1)$-term tends to zero as $n\to \infty$.
\end{thm}

\subsection{Distinct distances}

The equally famous Erd\H{o}s distinct distances problem is concerned with estimating the minimum possible number $D_{\|.\|_2}(n)$ of distinct distances determined by $n$ points in the Euclidean plane $\R^2$. See for example the book \cite{distinct_distances_book} focusing on this question, its history, generalisations and connections to other areas. The Erd\H{o}s distinct distances problem also traces its origins to the 1946 paper \cite{Erdos-unit-distance} of Erd\H{o}s. In this paper, he considered the same problem for higher-dimensional Euclidean spaces as well. For the planar case, Erd\H{o}s proved an upper bound of $O(n/\sqrt {\log n})$ and conjectured this is tight. After a long sequence of improvements of the lower bound, Guth and Katz established in \cite{GK} a nearly tight lower bound of $\Omega(n/ \log n)$. For higher-dimensional Euclidean spaces even the correct exponent of $n$ is not known, see \cite{SV2} for the best known bounds. 

The Erd\H{o}s distinct distances problem has a long history in the case of general $d$-norms as well, see \cite{Sw} for a recent survey on what is known in this direction. Given a $d$-norm $\|.\|$ let us denote by $D_{\|.\|}(n)$ the minimum possible number of distinct distances, according to $\|.\|$, determined by $n$ points in $\R^d$. For every $d$-norm $\|.\|$ and  every $n$, we clearly have $D_{\|.\|}(n) \leq n-1$ by considering any set of $n$ points along an arithmetic progression on a line. Brass conjectured that $D_{\|.\|}(n) = o(n)$ for any $d \geq 2$ and any $d$-norm $\|.\|$ (see \cite[Chapter 5.4, Conjecture 5, p.\ 211]{BMP-survey}). Here we refute his conjecture in a strong form. For most $d$-norms $\|.\|$, we show that $D_{\|.\|}(n)$ is not only linear in $n$, but is in fact of the form $(1-o(1))n$. 

\begin{thm}
\label{thm:distinct-distances} 
For any fixed $d \ge 2$, for most $d$-norms $\|.\|$ we have
\[D_{\|.\|}(n) =(1-o(1))n\]
for all $n$. More precisely, for all $d$-norms $\|.\|$ besides a meagre set, the following holds: For every $n$, among any $n$ points in $\R^d$ there are at least $(1-o(1))n$ distinct distances according to $\|.\|$, where the
$o(1)$-term tends to zero as $n\to \infty$.
\end{thm}

\subsection{Hadwiger--Nelson Problem}
The question of determining the chromatic number of the unit distance graph of the Euclidean norm in the plane is yet another famous open problem in discrete geometry, known as the Hadwiger--Nelson problem. In other words, this question asks how many colours are needed in order to colour all points in the plane such that any two points with Euclidean distance one receive different colours. This problem dates back to 1950, and it has been known for a long time that the answer is at least $4$ and at most $7$. In a recent computational breakthrough \cite{Exoo-Ismailescu,de-Grey}, the lower bound has been improved to $5$, sparking a collaborative Polymath project \cite{polymath_project} focusing on the problem. See also \cite{BMP-survey,Sw} for more details on the history of the Hadwiger--Nelson problem.

The analogue of this problem for other planar norms was studied by Chilakamarri \cite{Chilakamarri} in 1991, who showed that the unit distance graph of every $2$-norm has chromatic number at least $4$ and at most $7$. In addition, \cite[Problem 5]{Chilakamarri}, attributed to Robertson, asks for the chromatic number of the unit distance graph of at least one strictly convex $2$-norm to be evaluated. We prove that the chromatic number of the unit distance graph of most $2$-norms equals $4$.

\begin{thm}\label{thm:hadwiger-nelson}
For all $2$-norms $\|.\|$ besides a meagre set, the unit distance 
graph of $\|.\|$ has chromatic number equal to $4$.
\end{thm}

In contrast, as mentioned above, it was recently shown \cite{Exoo-Ismailescu,de-Grey} that for the Euclidean $2$-norm the chromatic number of the unit distance graph is at least $5$. Thus, the behaviour for most $2$-norms is different from the Euclidean $2$-norm, which in particular disproves a conjecture of Chilakamarri \cite[p.~355]{Chilakamarri}. 

In dimension $d$, our arguments give an upper bound of $2^d$ for the chromatic number of the unit distance graph of most $d$-norms. It is known that the chromatic number of the unit distance graph of any $d$-norm is at most exponential in $d$ (see \cite{FK,Ku}), but the exponential base in these known bounds is much larger than $2$. See also \cite{Sw} for a more detailed survey on what is known surrounding this problem. 
We furthermore note that results of Frankl and Wilson \cite{FW} give an exponential lower bound for the chromatic number of the unit distance graph of all $d$-norms which are invariant under coordinate permutations.

Our upper bound of $2^d$ for most $d$-norms actually also holds for the chromatic number of the odd distance graph (i.e.\ the graph whose edges correspond to pairs of points in $\R^d$ whose distance is an odd integer), which is a stronger result as the unit distance graph is a subgraph of the odd distance graph. This is very different from the Euclidean case, where by a recent result of Davies~\cite{Da} already in the plane the chromatic number of the odd distance graph is infinite.

\section{Proof overview}\label{Sec:outline}
The proofs of our theorems use arguments from combinatorics, polyhedral and discrete geometry, topology and algebra. In this section, we give a high-level proof overview. We note that this overview is simplifying numerous parts and omitting certain important points, since the goal of this section is to present just the general ideas and how they come together. We begin with \Cref{thm:main}. 

The proof of \Cref{thm:main} splits into two parts. The first, more combinatorial part shows that for a norm with a lot of unit distances among some set of $n$ points there must be a lot of rational linear dependencies between the corresponding unit vectors. More precisely, if the bound in \Cref{thm:main} fails for some $d$-norm $||.||$, namely if we can find $n$ points with more than $(d/2)\cdot n\log_2 n$ unit distances, then there is some set of unit vectors whose $\mathbb{Q}$-linear span contains a lot of other unit vectors. The second, more geometric part consists of showing that for most norms, this situation cannot happen. In other words, for most norms, there cannot be a collection of unit vectors from which a lot of other unit vectors can be obtained as rational linear combinations.

In the first part, we fix $n$ points with more than $(d/2)\cdot n\log_2 n$ unit distances according to some some $d$-norm $||.||$. Our goal is to show that there must be a lot of rational linear dependencies between the vectors $\mathbf{u}_1,\dots, \textbf{u}_k$ with $||\mathbf{u}_1||=\dots =||\textbf{u}_k||=1$ describing these unit distances. If there are not so many rational linear dependencies (more precisely, if for every $\ell$, the $\mathbb{Q}$-linear span of any $\ell$ of the vectors $\mathbf{u}_1,\dots, \textbf{u}_k$ contains at most $d\cdot \ell$ of the vectors $\mathbf{u}_1,\dots, \textbf{u}_k$), then we can find a subset of these unit vectors that is linearly independent over $\mathbb{Q}$ and accounts for a relatively large number of unit distances among the $n$ fixed points (namely, for more than $(1/2)\cdot n\log_2 n$ unit distances). On the other hand, for $n$ fixed points, we can analyse for how many pairs of points their difference can be among a given list of $\mathbb{Q}$-linearly independent vectors, and it turns out that this can happen for at most $(1/2)\cdot n\log_2 n$ pairs of points 
(this follows from an isoperimetric inequality for grids due to Bollob\'as and Leader \cite{BL}, but we also give a self-contained proof). 
This contradiction means that our original set of unit vectors $\mathbf{u}_1,\dots, \textbf{u}_k$  must have many rational linear dependencies, as desired.

For the second part, we need to show that there are only ``few'' special $d$-norms with the property that for some collection of unit vectors, a lot of other unit vectors can be obtained as rational linear combinations. So we need to show that the set of $d$-norms with this property can be covered by countably many nowhere-dense subsets in the space of $d$-norms (with a certain natural topology, see \Cref{S2} for more details). Roughly speaking, our countably many subsets correspond to the different possibilities for the size of the collection of unit vectors and the rational coefficients in the relevant rational linear combinations, as well as a lower bound for the separation angle between any two of these unit vectors. To prove that each of these subsets is nowhere dense in the space of all $d$-norms, we show that for any $d$-norm whose unit ball is a convex polytope with sufficiently small facets, there is a slight perturbation (obtained by slightly translating each facet hyperplane) that does not appear as the unit ball of an accumulation point of the subset. Heuristically speaking, if such a perturbation appears as the unit ball of a norm from this subset, then certain linear equations need to be satisfied for the translation lengths for the facets containing the unit vectors appearing in the above property (as a consequence of the linear relations between the unit vectors themselves). A generic perturbation does not satisfy these linear equations and therefore cannot correspond to an accumulation point of norms from this subset. This establishes that each of these subsets is indeed nowhere dense, and concludes the proof of \Cref{thm:main}.

The proof of \Cref{thm:distinct-distances} follows a similar overall approach, but various steps become more involved and significant new ideas are required. In the first part, for $n$ given points with few distinct distances according to some $d$-norm, we work over the field extension of $\mathbb{Q}$ generated by the distinct distances between these $n$ points (so we consider linear relations with coefficients in this extension field). We also need to establish stronger quantitative bounds for the number of unit vectors that we find in the span of some collection of unit vectors. Overall, the first part of the proof of \Cref{thm:distinct-distances} is much more involved than for \Cref{thm:main}. For example, one cannot use the Bollob\'as--Leader edge-isoperimetric inequality anymore in this setting, since here we need a more general statement (see \Cref{lem:point-sets-lin-independent}) compared to the setting of \Cref{thm:main} (we prove this more general statement via combinatorial arguments). The second part of the proof of \Cref{thm:distinct-distances} is relatively similar to \Cref{thm:main}, but requires some basic tools from algebraic geometry instead of the linear algebraic arguments in the proof of \Cref{thm:main}.

For our result on the Hadwiger--Nelson problem for typical norms in \Cref{thm:hadwiger-nelson}, we combine the second part of the proof of \Cref{thm:main} with a modified version of the first part. In particular, we make crucial use of the Edmonds Matroid Partitioning Theorem \cite{edmonds}.

We note that the general idea of trying to control the number of linear dependencies between unit vectors in a typical norm is due to Matou\v{s}ek \cite{matousek}, and his proof of the upper bound $U_{\|.\|}(n) \leq O(n \log n \log \log n)$ for most $2$-norms $||.||$ similarly splits into two parts as our proof of \Cref{thm:main}. For the second part, showing that for most norms there cannot be many linear dependencies between unit vectors, our approach is inspired by his arguments. 
However, he uses hands-on geometric arguments in dimension $2$ about ``bulging'' line segments or decomposing into trapezoids, that do not seem to easily extend to higher dimensions. Therefore, to obtain our results in any dimension, we introduced several new ideas, in particular making the proof more algebraic. 
On the other hand, the first part of our proof is radically different from Matou\v{s}ek's. His proof is based on graph theoretic ``expansion'' arguments via probabilistic methods, leading to his weaker bound. We develop a completely new strategy for the first part of the proof instead, resulting in an essentially tight bound.

To prove our new lower bound on $U_{\|.\|}(n)$ for any $d$-norm $||.||$ in \Cref{t12}, we consider Minkowski sums of certain carefully chosen small point sets, such that in a grid-like fashion we obtain many unit distances. To ensure that we can find suitable small point configurations serving as the base sets of our construction, we use a result from dimension theory: the Hurewicz dimension lowering theorem.

\textbf{Organisation.} The next section contains some background and preliminary lemmas. The first part of the proofs of \Cref{thm:main,thm:distinct-distances} can be found in Section \ref{S3}, and the second part in Section \ref{S4}. \Cref{thm:hadwiger-nelson} is then proved in \Cref{sec:hadwiger-nelson}, relying on the previous sections. Finally, we prove Theorem \ref{t12} in Section \ref{S5} and finish with some concluding remarks and open problems in Section \ref{S6}.

\section{Geometric preliminaries}
\label{S2}

\subsection{Background}

We begin by setting up some notation and introducing the 
notions we will work with. We note that all our logarithms are in base $2$ unless specified otherwise.

For $d\ge 1$, a norm on $\R^d$ is a mapping $\nn$ that assigns a non-negative real number $||\textbf{x}||$ to each $\textbf{x} \in \R^d$ such that the following three conditions hold:
\begin{compactitem}
\item For every $\textbf{x}\in \R^d$, we have $||\textbf{x}||= 0$ if and only if $\textbf{x} = \textbf{0}$.
\item We have $||\alpha \textbf{x}||= |\alpha|\cdot||\textbf{x}||$ for all $\alpha \in \R$ and $\textbf{x}\in \R^d$.
\item The triangle inequality holds, meaning that $||\textbf{x} + \textbf{y}||\le  ||\textbf{x}|| + ||\textbf{y}||$ for all $\textbf{x}, \textbf{y}\in \R^d$.
\end{compactitem}
Each norm $\nn$ on $\R^d$ is uniquely specified by its \emph{unit
ball}, defined as the set of all $\textbf{x}\in \R^d$ for which
$||\textbf{x}||\le 1$. A unit ball of any norm is a closed, bounded,
$\textbf{0}$-symmetric convex body containing $\textbf{0}$ in its
interior. Furthermore, any such body appears as the unit ball of a unique
norm. Let $\B_d$ denote the set of all unit balls of norms in $\R^d$
or equivalently the set of all closed, bounded, $\textbf{0}$-symmetric
convex bodies in $\R^d$ containing $\textbf{0}$ in the interior. As
discussed below, it is known that $\B_d$
endowed with the so-called Hausdorff metric $d_H$ forms a Baire space.

The Hausdorff distance $d_H(A,B)$ of two sets $A,B\su \R^d$ is defined as
\[d_H(A,B):=\max\left\{\sup_{\textbf{a} \in A} \inf_{\textbf{b} \in B} ||\textbf{a}-\textbf{b}||_2,\: \sup_{\textbf{b} \in B} \inf_{\textbf{a} \in A} ||\textbf{a}-\textbf{b}||_2\right \},\]
where $\nn_2$ denotes the Euclidean distance in $\R^{d}$. If $A,B\su \R^d$ are closed and bounded, then one can replace the suprema and infima in the above definition with maxima and minima. So in this case the Hausdorff distance $d_H(A,B)$ is simply the ``maximum distance'' of a point in $A$ from the set $B$ or of a point in $B$ from the set $A$. Note that the Hausdorff distance satisfies the triangle inequality.

A set $S$ in a metric (or topological) space $X$ is nowhere dense if every non-empty open set $U \subseteq X$ contains a nonempty open set $V$ with $V \cap S=\emptyset$. A meagre set in $X$ is a countable union of nowhere dense sets. Note that a subset of any meagre set is also meagre. The space $X$ is called a Baire space if the complement of each meagre set in $X$ is dense. It is known that $\B_d$ endowed with the Hausdorff metric $d_H$ forms a Baire space (this follows for example from \cite[Theorem 6.4]{gruber} together with the Baire Category Theorem). 

The diameter of a non-empty bounded closed subset $S\su \mathbb{R}^d$ (with respect to the Euclidean distance) is defined as
\[\diam(S)=\max_{\textbf{a},\textbf{b}\in S} ||\textbf{a}-\textbf{b}||_2\]
(note that this maximum is indeed well-defined since $S$ is closed and bounded).

A half-space in $\mathbb{R}^d$ is the closed subset of $\mathbb{R}^d$ given by the solutions $\textbf{x}\in \mathbb{R}^d$ to some linear inequality of the form $\textbf{a}\cdot \textbf{x}\le b$ for some $\textbf{a}\in \mathbb{R}^d$ and some $b\in \mathbb{R}$ (geometrically, this is the set of points on one side of the affine hyperplane given by $\textbf{a}\cdot \textbf{x}= b$, including the hyperplane itself). A polytope $P\su \R^d$ is an intersection of finitely many half-spaces in $\mathbb{R}^d$. Note that every polytope $P$ is convex. Every bounded polytope $P\su \R^d$ can also be described as a convex hull of finitely many points in $\R^d$.

Every $\textbf{0}$-symmetric polytope $P$ containing $\textbf{0}$ in its interior can be written in the form
\[P=\{\textbf{x}\in \R^d \midd |\textbf{o}_i\cdot \textbf{x}|\le t_i\text{ for }i=1,\dots,h\}\]
with non-zero vectors $\textbf{o}_1,\dots,\textbf{o}_h\in \R^d$ and positive real numbers $t_1,\dots,t_h$. The facets of such a polytope are the intersections of the form $P\cap H$ for some hyperplane $H$ such that the intersection $P\cap H$ is $(d-1)$-dimensional and $P$ is contained in one of the closed half-spaces bounded by $H$. If $P$ is $\textbf{0}$-symmetric, then the facets appear in pairs of opposite facets (which are parallel to each other).

A set $B\su \R^d$ is strictly convex, if for all distinct $\textbf{a},\textbf{b}\in B$ and all $0<\alpha<1$ the point $\alpha \textbf{a}+(1-\alpha)\textbf{b}$ is in the interior of $B$. For a strictly convex set $B\su \R^d$, for every point $\textbf{b}$ on the boundary of $B$ there exists a hyperplane $H$ with $H\cap B=\{\textbf{b}\}$ such that $B$ is contained in one of the half-spaces bounded by $H$. The unit ball $B$ of a norm $\nn$ on $\R^d$ is strictly convex if the triangle inequality is a strict inequality $||\textbf{x} + \textbf{y}||<  ||\textbf{x}|| + ||\textbf{y}||$ for all non-zero vectors $\textbf{x}, \textbf{y}\in \R^d$ that are not multiples of each other. 
Indeed, in this case for non-zero $\textbf{a},\textbf{b}\in B$ with $\spn_{\R}(\textbf{a})\ne \spn_{\R}(\textbf{b})$ and $0<\alpha<1$, we have  $||\alpha \textbf{a}+(1-\alpha)\textbf{b}||<||\alpha \textbf{a}||+||(1-\alpha)\textbf{b}||=\alpha ||\textbf{a}||+(1-\alpha)||\textbf{b}||\le \alpha+(1-\alpha)=1$. For distinct vectors $\textbf{a},\textbf{b}\in B$ with $\spn_{\R}(\textbf{a})= \spn_{\R}(\textbf{b})$ we always have $||\alpha \textbf{a}+(1-\alpha)\textbf{b}||< 1$ since we either have $\max\{||\textbf{a}||,||\textbf{b}||\}<1$ or $\textbf{a}=-\textbf{b}$. And for $\textbf{a},\textbf{b}\in B$ with $\textbf{a}=\textbf{0}$ or $\textbf{b}=\textbf{0}$ we also trivially have $||\alpha \textbf{a}+(1-\alpha)\textbf{b}||< 1$. Thus the triangle inequality for $\nn$ being strict for all non-zero vectors $\textbf{x}, \textbf{y}\in \R^d$ with $\spn_{\R}(\textbf{x})\ne \spn_{\R}(\textbf{y})$ indeed implies that $B$ is strictly convex.

Finally, we record the following simple algebraic fact, which we will use in our proof of Theorem \ref{thm:distinct-distances}.

\begin{fact}\label{fact-trans-degree}
For any positive integer $m$, given $m+1$ rational functions $f_1,\dots,f_{m+1}\in \mathbb{R}(x_1,\dots,x_m)$ in $m$ variables with real coefficients, there exists a nonzero polynomial $P\in \mathbb{R}[y_1,\dots,y_{m+1}]$ such that $P(f_1,\dots,f_{m+1})=0$.
\end{fact}

This follows immediately from the well-known fact that the transcendence degree of the function field \linebreak $\mathbb{R}(x_1,\dots,x_m)$ over $\mathbb{R}$ is equal to $m$. Indeed, since $m+1$ is larger than this transcendence degree, there must be an algebraic relationship between $f_1,\dots,f_{m+1}$.

\subsection{Geometric lemmas}

This section contains some basic geometric lemmas. Although most of the content of this section is well-known, we include some of the proofs for the reader's convenience.

The first two lemmas below are only needed to prove the third lemma in this section.

\begin{lem}\label{lem-eps-net}
    Let $T\su \R^d$ be a bounded subset and let $\eps>0$. Then there exists a finite subset $S\su T$ such that for every point $\textbf{t}\in T$ there exists a point $\textbf{s}\in S$ with $||\textbf{s}-\textbf{t}||_2\le \eps$.
\end{lem}

We remark that such a subset $S$ is called an $\eps$-net of $T$.

\begin{proof}
    Since $T$ is bounded, it is contained in the Euclidean ball of radius $R$ around $\textbf{0}$ for some $R>0$.
    Let us consider the family of all subsets $S\su T$ with the property that $||\textbf{s}-\textbf{s}'||_2> \eps$ for all distinct $\textbf{s},\textbf{s}'\in S$. We claim that each such subset $S$ has size $|S|\le (2R/\eps+1)^d$. Indeed, the Euclidean balls of radius $\eps/2$ around the points in $S$ are mutually disjoint and contained in the Euclidean ball of radius $R+\eps/2$ around $\textbf{0}$. Hence, for volume reasons, the set $S$ can consist of at most $(2R/\eps+1)^d$ points.
    
   Note that $S=\emptyset$ vacuously satisfies the property that $||\textbf{s}-\textbf{s}'||_2> \eps$ for all distinct $\textbf{s},\textbf{s}'\in S$. Hence, the family of subsets $S\su T$ with this property is non-empty, and as all its members satisfy $|S|\le (2R/\eps+1)^d$, there must be a maximal subset $S\su T$ in this family. So let $S\su T$ be a maximal subset with the property that we have $||\textbf{s}-\textbf{s}'||_2> \eps$ for all distinct $\textbf{s},\textbf{s}'\in S$.
  
    Now, let us check that for every point $\textbf{t}\in T$ there exists a point $\textbf{s}\in S$ with $||\textbf{s}-\textbf{t}||_2\le \eps$. If $\textbf{t}\in S$, we can choose $\textbf{s}=\textbf{t}$. If $\textbf{t}\not\in S$, then by the maximality of our chosen set $S$ we cannot add $\textbf{t}$ to the set. This means that we must have $||\textbf{s}-\textbf{t}||_2\le \eps$ for some $\textbf{s}\in S$, as desired.
\end{proof}

The following lemma states that close to any $B\in \B_d$ we can find some strictly convex $B'$. This lemma can also be deduced from a classical result of Klee \cite{klee} from 1959 which says that almost all norms on $\R^d$ are strictly convex.

\begin{lem}\label{lem:approximate-by-strictly-convex}
For every $B\in \B_d$ and every $\mu>0$, there exists a strictly convex $B' \in \B_d$ such that $d_H(B,B')\le \mu$.
\end{lem}
\begin{proof}
Let us denote by $||.||$ the norm with unit ball $B$. Then $B=\{\textbf{x}\in \R^d \midd ||\textbf{x}||\le 1\}$. Since $B$ is bounded, we can choose some $c>0$ such that $||\textbf{x}||_2\le c$ for all $\textbf{b}\in B$. Let $\eps =\mu/c^2$.

Let us define a norm $||.||'$ by $||\textbf{x}||':=||\textbf{x}||+\eps||\textbf{x}||_2$ for all $\textbf{x}\in \R^d$. 
To check that $||.||'$ is indeed a norm, note that the triangle inequalities for $||.||$ and $||.||_2$ imply that
\[||\textbf{x}+\textbf{y}||'=||\textbf{x}+\textbf{y}||+\eps||\textbf{x}+\textbf{y}||_2\le ||\textbf{x}||+||\textbf{y}||+\eps||\textbf{x}||_2+\eps||\textbf{y}||_2= ||\textbf{x}||'+||\textbf{y}||',\]
for all $\textbf{x},\textbf{y}\in \R^d$, so the triangle inequality also holds for $||.||'$. Note in addition that we can have equality only if $||\textbf{x}+\textbf{y}||_2=||\textbf{x}||_2+||\textbf{y}||_2$ and since the triangle inequality is strict for $||.||_2$ for any non-zero vectors $\textbf{x},\textbf{y}\in \R^d$ which are not multiples of each other, the same also holds for $||.||'$. Thus, the unit ball $B'$ of the norm $||.||'$ is strictly convex.

Note that $B' \subseteq B$, so for any point $\textbf{b}'\in B'$ we have a point $\textbf{b}=\textbf{b}'\in B$ at Euclidean distance $0\le \mu$ from $\textbf{b}'$. On the other hand, we claim that for any point $\textbf{b}\in B$ we can also find a point $\textbf{b}'\in B'$ with Euclidean distance at most $\mu$ from $\textbf{b}$. If $||\textbf{b}||_2\le \mu$, we can simply take $\textbf{b}'=\textbf{0}\in B'$. If $||\textbf{b}||_2> \mu$, we have
\[\left\Vert\left(1-\frac{\mu}{||\textbf{b}||_2}\right) \textbf{b}\right\Vert'=\left(1-\frac{\mu}{||\textbf{b}||_2}\right) ||\textbf{b}||'=\left(1-\frac{\mu}{||\textbf{b}||_2}\right) (||\textbf{b}||+\eps ||\textbf{b}||_2) \le (1-\eps||\textbf{b}||_2) (1+\eps ||\textbf{b}||_2)<1,\]
using $ \eps =\mu/c^2\le \mu/(||\textbf{b}||_2)^2$ and $||\textbf{b}||\le 1$. So we can take $\textbf{b}':=\left(1-\frac{\mu}{||\textbf{b}||_2}\right) \textbf{b} \in B'$ and have $||\textbf{b}'-\textbf{b}||_2=\frac{\mu}{||\textbf{b}||_2}\cdot ||\textbf{b}||_2=\mu$. This shows that $d_H(B,B')\le\mu$, as desired.
\end{proof}

We will need the following lemma which tells us that we can approximate any $B \in \B_d$ with a polytope in $\B_d$ with small facets. We note there is plenty of research concerned with similar polytope approximation problems, see e.g.\ \cite{polytope-approx} and references therein, mostly concerned with minimising the number of facets needed to obtain a good approximation. Here, we are not concerned with the number of facets of the approximating polytope, but we need all of the facets to have small diameter. 

\begin{lem}\label{lem:approximating-polytope}
    For every $B\in \B_d$ and every $\mu>0$, there exists a bounded $\textbf{0}$-symmetric polytope $B'\in \B_d$ containing $\textbf{0}$ in its interior such that $d_H(B,B')<\mu$ and all facets of $B'$ have diameter less than $\mu$ (with respect to the Euclidean distance).
\end{lem}

Note that we can endow the set of all $(d-1)$-dimensional hyperplanes in $\mathbb{R}^d$ with the natural topology induced by $\mathbb{R}^{d+1}$ when identifying a hyperplane described by an equation of the form $\textbf{a}\cdot \textbf{x}= b$ with $||\textbf{a}||_2=1$ with the point $(\textbf{a},b)$. More precisely, such a hyperplane with an equation of the form $\textbf{a}\cdot \textbf{x}= b$ corresponds to a point in the quotient space $(\textbf{a},b)\in (S^{d-1}\times \mathbb{R})/\{\pm 1\}$ (here $S^{d-1}$ is the $(d-1)$-dimensional unit sphere in $\R^d$ and we consider the quotient by the action of $\{\pm 1\}$ on $(\textbf{a},b)\in S^{d-1}\times \mathbb{R}$, where $-1$ acts by sending $(\textbf{a},b)\mapsto (-\textbf{a},-b)$).

We denote by $\dist_2(\textbf{p},H)$ the (Euclidean) distance from a point $\textbf{p}\in \R^d$ to a hyperplane $H\su \R^d$.

\begin{proof}
    By \Cref{lem:approximate-by-strictly-convex}, there exists a strictly convex $B''\in \B_d$ with $d_H(B'',B)<\mu/2$. It is now sufficient to show that there exists a bounded $\textbf{0}$-symmetric polytope $B'\in \B_d$ containing $\textbf{0}$ in its interior such that $d_H(B',B'')<\mu/2$ and all facets of $B'$ have diameter less than $\mu$
    
    Let $\mathcal{H}$ be the set of all $(d-1)$-dimensional hyperplanes $H$ in $\mathbb{R}^d$ such that $H\cap B''$ has diameter at least $\mu$. Note that $\mathcal{H}$ is a closed and bounded subset of the set of all $(d-1)$-dimensional hyperplanes in $\mathbb{R}^d$ (to see that $\mathcal{H}$ is closed, note that $\diam(H \cap B'')$ is a continuous function on the set of all $(d-1)$-dimensional hyperplanes $H\su \R^d$). Thus, $\mathcal{H}$ is compact in the induced topology. 
    
    Every hyperplane $H\in \mathcal{H}$ cuts the strictly convex set $B''$ into two parts $B''_1(H)$ and $B''_2(H)$ (strictly speaking these are the intersections of $B''$ with the two closed half-spaces bounded by $H$), neither of which is contained in $H$, by the strict convexity assumption on $B''$. For every hyperplane $H\in \mathcal{H}$, define
    \[f(H)=\min\left\{\max_{\textbf{b}\in B''_1(H)}\dist_2(\textbf{b},H),\max_{\textbf{b}\in B''_2(H)}\dist_2(\textbf{b},H)\right\},\]
    and note that $f(H)>0$ for every $H\in \mathcal{H}$. Now $f$ is a continuous function on a compact space and so it attains a minimum. So let $\eta>0$ be this minimum, then $f(H)\ge \eta$ for all $H\in \mathcal{H}$.

    Now, let us apply Lemma \ref{lem-eps-net} to the set $B''$ and $\eps=\min\{\mu,\eta\}/2$, i.e.\ let us choose an $\eps$-net $S\su B''$ of $B''$. By adding up to $d+1$ additional points to $S$, we may assume that $\textbf{0}$ is in the interior of the convex hull of $S$.

    We may furthermore assume that $S$ is $\textbf{0}$-symmetric, meaning that $-S=S$ (indeed, for every point $\textbf{s}\in S$ we may add the point $-\textbf{s}$ to $S$ if it is not already contained in $S$).
    
    Now, for every hyperplane $H\in \mathcal{H}$ there exist points in $S$ on both sides of $H$ (more precisely, the two open half-spaces bounded by $H$ both contain at least one point in $S$). Indeed, since $f(H)\ge \eta$, there exist points $\textbf{b}_1\in B''_1(H)\su B''$ and $\textbf{b}_2\in B''_2(H)\su B''$ with $\dist_2(\textbf{b}_1,H)\ge \eta$ and $\dist_2(\textbf{b}_2,H)\ge \eta$. Note that $\textbf{b}_1$ and $\textbf{b}_2$ are on opposite sides of the hyperplane $H$ (and not on $H$ itself). Now, we can choose points $\textbf{s}_1,\textbf{s}_2\in S$ with $||\textbf{s}_1-\textbf{b}_1||\le \eps\le\eta/2$ and $||\textbf{s}_2-\textbf{b}_2||\le \eps\le\eta/2$. Then $\textbf{s}_1$ must be on the same side of $H$ as $\textbf{b}_2$ (and not on $H$ itself), and similar for $\textbf{s}_2$. Thus, $\textbf{s}_1$ and $\textbf{s}_2$ lie on opposite sides of $H$ in the two open half-spaces bounded by $H$. So indeed for every $H\in \mathcal{H}$ the two open half-spaces bounded by $H$ both contain at least one point in $S$.

    Finally, let us define $B'=\operatorname{conv}(S)$ to be the convex hull of $S$. By our assumptions on $S$, the set $B'$ is a  bounded $\textbf{0}$-symmetric polytope containing $\textbf{0}$ in its interior. In particular, $B'$ is $d$-dimensional

    To check that $d_H(B'',B')\le \mu/2$, first note that $B'=\operatorname{conv}(S)\su \operatorname{conv}(B'')=B''$ (as $B''$ is a convex set). So we have $\sup_{\textbf{b}'\in B'}\inf_{\textbf{b}\in B''} ||\textbf{b}-\textbf{b}'||_2=0<\mu$. Furthermore, for every $\textbf{b}\in B''$ there is a point $\textbf{b}'\in S\su B'$ with $||\textbf{b}-\textbf{b}'||_2\le \eps\le \mu/2$ and hence $\sup_{\textbf{b}\in B''}\inf_{\textbf{b}'\in B'} ||\textbf{b}-\textbf{b}'||_2\le \mu/2$. This shows that indeed $d_H(B'',B')\le \mu/2$.

    Finally, it remains to check that the facets of $B'$ all have diameter less than $\mu$. So suppose that $B'$ had a facet of diameter at least $\mu$. Then for the $(d-1)$-dimensional hyperplane $H$ through this facet, the set $H\cap B''$ (which contains this facet) has diameter at least $\mu$. Hence $H\in \mathcal{H}$, but this means that the two open half-spaces bounded by $H$ both contain at least one point in $S\su B'$. This is a contradiction to the fact that $H$ is a hyperplane through a facet of $B'$. Hence the facets of $B'$ indeed all have diameter less than $\mu$.
\end{proof}

\begin{lem}\label{lem:Hausdorff-distance-boundary}
    Let $\delta>0$ and let $B,B'\in \B_d$ be such that $d_H(B,B')\le \delta$. Then for every point $\textbf{x}$ on the boundary of $B$ there exists a point $\textbf{y}$ on the boundary of $B'$ with $||\textbf{x}-\textbf{y}||_2\le \delta$.
\end{lem}

\begin{proof}
We distinguish three cases depending on the position of $\textbf{x}$ in relation to $B'$. First, if $\textbf{x}$ is on the boundary of $B'$, we can take $\textbf{y}=\textbf{x}$ to satisfy the statement in the lemma.

Next, assume that $\textbf{x}\not\in B'$. Then, as $d_H(B,B')\le \delta$, we have $\min_{\textbf{y}'\in B'}||\textbf{x}-\textbf{y}'||_2=\inf_{\textbf{y}'\in B'}||\textbf{x}-\textbf{y}'||_2\le \delta$, and so there exists a point $\textbf{y}'\in B'$ with $||\textbf{x}-\textbf{y}'||_2\le \delta$. Now, consider the straight line segment from $\textbf{y}'\in B'$ to $\textbf{x}\not\in B'$. This straight line segment must contain some point $\textbf{y}$ on the boundary of $B'$, and we have $||\textbf{x}-\textbf{y}||_2\le ||\textbf{x}-\textbf{y}'||_2\le\delta$.

Finally, assume that $\textbf{x}$ is in the interior of $B'$. Recalling that $\textbf{x}$ is on the boundary of $B$, let $H$ be a supporting hyperplane of $B$ through $\textbf{x}$ (this means that $B$ is contained in one of the closed half-spaces bounded by $H$ and $\textbf{x}$ is on $H$). Now consider the ray orthogonal to $H$, starting at $\textbf{x}$, pointing away from $B$. This ray needs to contain a point $\textbf{y}$ on the boundary of $B'$ (since the start of the ray at $\textbf{x}$ is in the interior of $B'$ and $B'$ is bounded). If $||\textbf{x}-\textbf{y}||_2> \delta$, then the closed Euclidean ball of radius $\delta$ around $\textbf{y}$ is disjoint from the half-space bounded by $H$ containing $B$. Hence $\inf_{\textbf{b}\in B}||\textbf{b}-\textbf{y}||_2=\min_{\textbf{b}\in B}||\textbf{b}-\textbf{y}||_2>\delta$, which contradicts our assumption $d_H(B,B')\le \delta$. So we must have $||\textbf{x}-\textbf{y}||_2\le \delta$.
\end{proof}

\section{Point sets with many special differences}
\label{S3}

The following two lemmas encapsulate the first part of our proofs of \Cref{thm:main,thm:distinct-distances}, respectively (see also the proof overview in \Cref{Sec:outline}). The first of these lemmas states, roughly speaking, that for a list of vectors $\mathbf{u}_1,\dots, \textbf{u}_k$ in some vector space $V$ over $\mathbb{Q}$ and any subset of $V$ where many differences are in the set $\{\pm \mathbf{u}_1, \dots, \pm \textbf{u}_k\}$, there must be a lot of linear dependencies among the vectors $\mathbf{u}_1,\dots, \textbf{u}_k$. More precisely, the span of some small subset of the vectors $\mathbf{u}_1,\dots, \textbf{u}_k$ must contain many other vectors among $\mathbf{u}_1,\dots, \textbf{u}_k$.

When we apply \Cref{lem:point-sets-key-lemma} in the next section to complete the proof of \Cref{thm:main}, we will take $\mathbf{u}_1,\dots, \textbf{u}_k$ to be the unit vectors appearing as unit distances according to some norm in a point set in $\R^d$. If there are many such unit distances, then the lemma implies that the span of some small subset of the vectors $\mathbf{u}_1,\dots, \textbf{u}_k$ must contain many other vectors among $\mathbf{u}_1,\dots, \textbf{u}_k$. In \Cref{S4} we show (as the second part of our proof) that this is a ``special'' property, meaning that most norms cannot have this property (more precisely, the set of norms with this property is a meagre set).

\begin{lem}\label{lem:point-sets-key-lemma}
    Let $V$ be a vector space over $\mathbb{Q}$, and let $\mathbf{u}_1,\dots, \mathbf{u}_k\in V$ be non-zero vectors in $V$.  Let $\mathbf{p}_1,\dots,\mathbf{p}_n\in V$ be distinct vectors, and let us consider the graph with vertex set $\{1,\dots,n\}$, where for any $x,y\in \{1,\dots,n\}$ we draw an edge between the vertices $x$ and $y$ if and only if $\mathbf{p}_x-\mathbf{p}_y\in \{\pm \mathbf{u}_1, \dots, \pm \mathbf{u}_k\}$. For some positive integer $d$, suppose that this graph has more than $\frac d2\cdot n \log n$ edges. Then there exists a subset $I\su \{1,\dots,k\}$, such that we have $\mathbf{u}_\ell\in \spn_{\mathbb{Q}}(\mathbf{u}_i\midd i\in I)$ for at least $d\cdot |I|+1$ indices $\ell\in \{1,\dots,k\}$.
\end{lem}

We will prove Lemma \ref{lem:point-sets-key-lemma} later in this section, and then use it in the next section to prove \Cref{thm:main}. The following lemma will in a similar fashion be used to prove \Cref{thm:distinct-distances}. When applying this lemma, we will take $F$ to be the field extension of $\mathbb{Q}$ generated by all the distinct distances appearing in a given set of $n$ points according to some given norm.

\begin{lem}\label{lem:point-sets-lemma-for-distinct-distances}    
    Let $d\ge 1$ be an integer and let $0<\mu<1$. Suppose that $n$ is sufficiently large with respect to $d$ and $\mu$. Let $F\su \mathbb{R}$ be a subfield of $\mathbb{R}$, and let $V$ be a vector space over $\mathbb{R}$. Let $\mathbf{u}_1,\dots, \mathbf{u}_k\in V$ be non-zero vectors in $V$, and let $\mathbf{p}_1,\dots,\mathbf{p}_n\in V$ be distinct vectors such that not all of $\mathbf{p}_1,\dots,\mathbf{p}_n$ are lying on a common affine line in $V$ (as a vector space over $\mathbb{R}$). Suppose that for all $x,y\in \{1,\dots,n\}$ we have $\mathbf{p}_x-\mathbf{p}_y\in \spn_F(\mathbf{u}_i)$ for some $i\in \{1,\dots,k\}$. Then there exists a subset $I\su \{1,\dots,k\}$, such that we have $\mathbf{u}_\ell\in \spn_{F}(\mathbf{u}_i\midd i\in I)$ for at least $d\cdot |I|+(1-\mu)\cdot n+1$ indices $\ell\in \{1,\dots,k\}$.
\end{lem}

The proofs of \Cref{lem:point-sets-key-lemma,lem:point-sets-lemma-for-distinct-distances} rely on the following lemma, which we prove first.

\begin{lem}\label{lem:point-sets-lin-independent}
    Let $V$ be a vector space over a field $F$, and let $\mathbf{u}_1,\dots, \mathbf{u}_k\in V$ be linearly independent vectors in $V$. For some $n\ge 1$, let $\mathbf{p}_1,\dots,\mathbf{p}_n\in V$ be distinct vectors, and let $G$ be a graph with vertex set $\{1,\dots,n\}$ satisfying the following two conditions
    \begin{compactitem}
        \item For every edge $xy$ of the graph $G$, we have $\mathbf{p}_x-\mathbf{p}_y\in \spn_F(\mathbf{u}_i)$ for some $i\in \{1,\dots,k\}$.
        \item For each $i\in \{1,\dots,k\}$, the subgraph of $G$ consisting of all edges $xy$ with $\mathbf{p}_x-\mathbf{p}_y\in \spn_F(\mathbf{u}_i)$ is a forest.
    \end{compactitem}
 Then the graph $G$ has at most $\frac12 \cdot n \log n$ edges.
\end{lem}

In the proof of this lemma, we will use the following simple inequality.

\begin{lem}\label{lem:entropy-like-inequality}
    For any positive integers $n_1\ge n_2\ge \dots\ge n_\ell\ge 1$ with $n=n_1+\dots+n_\ell$ (and $\ell\ge 1$), we have
    \[n_1-\frac{1}{2}\sum_{i=1}^{\ell} n_i\log n_i\ge n-\frac{1}{2}\cdot n\log n.\]
\end{lem}
\begin{proof}
Recall that the binary entropy function given by $H(t)= -t\log t-(1-t)\log (1-t)$ for $t\in (0,1)$ and $H(0)=H(1)=0$ is a concave function on the interval $[0,1]$. It intersects the line given by $2-2t$ at $t=\frac12$ and $t=1$, and so for $\frac12 \le t\le 1$ we have $H(t)\ge 2-2t$ and therefore $t+\frac{1}{2}H(t)\ge 1$.

Let us now prove the desired statement by induction on $\ell$. For $\ell=1$, we have $n_1=n$ and the statement is trivially true. So let us assume that $\ell\ge 2$ and that we already proved the lemma for $\ell-1$. Now, setting $n_1'=n_1+n_\ell\le 2n_1$, we have $\frac12 \le n_1/n_1'\le 1$ and therefore
\begin{align*}
n_1-\frac{1}{2}\cdot n_1\log n_1-\frac{1}{2}\cdot n_\ell\log n_\ell&=n_1-\frac{1}{2}\cdot n_1' \cdot \log n_1'-\frac{1}{2}\cdot n_1\cdot \log(n_1/n_1')-\frac{1}{2}\cdot n_\ell\cdot \log(n_\ell/n_1')\\
&=-\frac{1}{2}\cdot n_1' \cdot \log n_1'+n_1'\cdot \left(\frac{n_1}{n_1'}-\frac{1}{2}\cdot\frac{n_1}{n_1'}\cdot \log(n_1/n_1')-\frac{1}{2}\cdot\frac{n_\ell}{n_1'}\cdot \log (n_\ell/n_1')\right)\\
&=-\frac{1}{2}\cdot n_1' \cdot \log n_1'+n_1'\cdot \left(\frac{n_1}{n_1'}+\frac{1}{2}\cdot H(n_1/n_1')\right)\ge -\frac{1}{2}\cdot n_1' \cdot \log n_1'+n_1'.
\end{align*}
Thus,
\[n_1-\frac{1}{2}\sum_{i=1}^{\ell} n_i\log n_i\ge n_1'-\frac{1}{2}\cdot n_1' \log n_1'-\sum_{i=2}^{\ell-1}n_i\log n_i\ge  n-\frac{1}{2}\cdot n\log n,\]
where the last inequality follows from the induction hypothesis applied to $n_1'\ge n_2\ge\dots\ge n_{\ell-1}$ (noting that $n_1'+ n_2+\dots+ n_{\ell-1}=(n_1+n_\ell)+n_2+\dots+ n_{\ell-1}=n$).
\end{proof}

Now, we are ready to prove \Cref{lem:point-sets-lin-independent}.

\begin{proof}[ of \Cref{lem:point-sets-lin-independent}]
    We will prove the lemma by induction on $n$. If $n=1$, then the graph has $0=(1/2)\cdot 1\cdot \log 1$ edges.

    So let us now assume that $n\ge 2$, and that we have already proved the lemma for all smaller values of $n$. Let $G$ be a graph as in the statement of the lemma.

    If the graph $G$ is not connected, then we can divide it into two disconnected parts with $n_1\ge 1$ and $n_2\ge 1$ vertices (where $n_1+n_2=n$). By the induction assumption these parts have at most $(1/2)\cdot n_1 \log n_1$ and $(1/2)\cdot n_2 \log n_2$ edges, respectively. Hence, the total number of edges in $G$ is at most
    \[\frac{1}{2}\cdot n_1 \log n_1+\frac{1}{2}\cdot n_2 \log n_2\le \frac{1}{2}\cdot (n_1+n_2) \log n=\frac{1}{2}\cdot n \log n,\]
    as desired. So let us from now on assume that $G$ is connected.
    
    If $G$ has no edges, the desired conclusion trivially holds, so let us assume that $G$ has at least one edge $xy$. Then there exists some index $i\in \{1,\dots,k\}$ with $\mathbf{p}_x-\mathbf{p}_y\in \spn_F(\mathbf{u}_i)$. Upon relabelling, we may assume that $i=k$, i.e.\ we may assume that there exists at least one edge $xy$ in $G$ with $\mathbf{p}_x-\mathbf{p}_y\in \spn_F(\mathbf{u}_k)$.

    Let us now colour any edge $xy$ in $G$ red if we have $\mathbf{p}_x-\mathbf{p}_y\in \spn_F(\mathbf{u}_k)$. Then there is at least one red edge in $G$. We remark that by the second condition on the graph $G$ in the statement of \Cref{lem:point-sets-lin-independent} the red edges form a forest.

    Let us fix an arbitrary vertex $z\in \{1,\dots,n\}$. Since $G$ is connected, every vertex $w\in \{1,\dots,n\}$ can be reached by some path in $G$ starting at $z$. For every edge $xy$ along this path, we have $\mathbf{p}_x-\mathbf{p}_y\in \spn_F(\mathbf{u}_i)$ for some $i\in \{1,\dots,k\}$. Adding up $\mathbf{p}_x-\mathbf{p}_y$ for all edges $xy$ along the path now gives a representation $\mathbf{p}_w=\mathbf{p}_z+a_1\mathbf{u}_1+\dots+a_k\mathbf{u}_k$ with coefficients $a_1,\dots,a_k\in F$. As the vectors $\mathbf{u}_1,\dots, \textbf{u}_k\in V$ are linearly independent over $F$, for every given $w\in \{1,\dots,n\}$ there is a unique such representation. In particular, for every $w\in \{1,\dots,n\}$, the vector $\mathbf{p}_w$ lies in exactly one of the sets $\mathbf{p}_z+a\mathbf{u}_k+F\mathbf{u}_1+\dots+F\mathbf{u}_{k-1}$ for $a\in F$.

    This gives a partition of the vertices $w\in \{1,\dots,n\}$ into subsets $W_a$ for $a\in F$, where for each $a\in F$ the set $W_a$ consists of those vertices $w\in \{1,\dots,n\}$ such that $\mathbf{p}_w\in \mathbf{p}_z+a\mathbf{u}_k+F\mathbf{u}_1+\dots+F\mathbf{u}_{k-1}$. Note that $\sum_{a\in F} |W_a|=n$.

     If $xy$ is an edge in $G$ which is not red, then $x$ and $y$ must belong to the same set $W_a$. Indeed, we have $\mathbf{p}_x-\mathbf{p}_y\in \spn_F(\mathbf{u}_i)$ for some $i\in \{1,\dots,k\}$ and so if $\mathbf{p}_x\in \mathbf{p}_z+a\mathbf{u}_k+F\mathbf{u}_1+\dots+F\mathbf{u}_{k-1}$, then we have $\mathbf{p}_y\in \mathbf{p}_z+a\mathbf{u}_k+F\mathbf{u}_1+\dots+F\mathbf{u}_{k-1}$ as well. Hence, every non-red edge of $G$ is inside one of the induced subgraphs $G[W_a]$ for $a\in F$ with $W_a\neq \emptyset$.

     Recall that the red edges of $G$ form a forest. We claim that for every $a\in F$ the vertices in the set $W_a$ must all be in distinct connected components of the red forest. Indeed, suppose towards a contradiction that for some $a\in F$ there are two distinct vertices $w,w'\in W_a$ belonging to the same component of the red forest. Then $w$ and $w'$ can be connected by a path of red edges, and for all edges $xy$ on this path we have $\mathbf{p}_x-\mathbf{p}_y\in\spn_F(\mathbf{u}_k)$. Adding this up over all edges on the path between $w$ and $w'$, we conclude that $\mathbf{p}_w-\mathbf{p}_{w'}\in\spn_F(\mathbf{u}_k)$. On the other hand, since $w,w'\in W_a$, we have $\mathbf{p}_w,\mathbf{p}_{w'}\in \mathbf{p}_z+a\mathbf{u}_k+F\mathbf{u}_1+\dots+F\mathbf{u}_{k-1}$ and hence $\mathbf{p}_w-\mathbf{p}_{w'}\in\spn_F(\mathbf{u}_1,\dots,\mathbf{u}_{k-1})$. Since the vectors $\mathbf{u}_1,\dots, \textbf{u}_k\in V$ are linearly independent over $F$, this implies that $\mathbf{p}_w-\mathbf{p}_{w'}=\mathbf{0}$, which is a contradiction (as the points $\mathbf{p}_1,\dots,\mathbf{p}_n$ are distinct). So indeed, for each $a\in F$, the vertices in $W_a$ belong to distinct connected components of the red forest.

     Recalling that there is at least one red edge in $G$, the two endvertices of this edge must be in different sets $W_a$ and $W_{a'}$. Thus, at least two of the sets $W_a$ for $a\in F$ are non-empty, and since $\sum_{a\in F} |W_a|=n$, we have $|W_a|<n$ for all $a\in F$.
     
     We can now apply the induction hypothesis to each of the graphs $G[W_a]$ for $a\in F$ with $W_a\neq \emptyset$, and obtain that $e(G[W_a])\le \frac 12 \cdot |W_a|\cdot \log |W_a|$
     whenever $W_a\neq \emptyset$. Hence, the number of non-red edges in $G$ is at most
     \[\sum_{\substack{a\in F\\W_a\neq \emptyset}} \frac{1}{2}\cdot |W_a|\cdot \log |W_a|=\frac{1}{2} \sum_{\substack{a\in F\\W_a\neq \emptyset}} |W_a|\cdot \log |W_a|.\]

     Let $a^*\in F$ be such that $|W_{a^*}|$ is of maximal size among the sets $|W_a|$ for $a\in F$. Then we have $|W_a|\le |W_{a^*}|$ for all $a\in F$.
     
     We claim that the number of red edges in $G$ is at most $n- |W_{a^*}|$. Indeed, the vertices in $W_{a^*}$ are all in distinct components of the forest formed by the red edges. Hence, this forest has at least $|W_{a^*}|$ components and can therefore have at most $n- |W_{a^*}|$ edges.

     Thus, the total number of edges in $G$ (both red edges and non-red edges) is bounded by 
     \[e(G)\le n- |W_{a^*}|+\frac{1}{2} \sum_{\substack{a\in F\\W_a\neq \emptyset}} |W_a|\cdot \log |W_a|\le n-\left(n-\frac 12 \cdot n \log n\right)=\frac 12 \cdot n \log n, \]
    where the second inequality follows from \Cref{lem:entropy-like-inequality}. 
\end{proof}

\begin{rem}
The assertion of \Cref{lem:point-sets-lin-independent} is tight for every $n$ which is a power of two, as shown by a generic embedding of the graph of the $k$-dimensional hypercube. Indeed, for any linearly independent vectors $\textbf{u}_1,\dots,\textbf{u}_k$, we can take $\textbf{p}_1,\dots,\textbf{p}_n$ with $n=2^k$ to be all subset sums of $\{\textbf{u}_1,\dots,\textbf{u}_k\}$. Now letting $G$ be the $k$-dimensional hypercube graph (with $\frac{1}{2}\cdot 2^k k=\frac{1}{2}\cdot n\log n$ edges), all conditions in \Cref{lem:point-sets-lin-independent} are satisfied (each of the forests in the second condition is a perfect matching). For the proof of Theorem \ref{thm:main} we only need a special case of  Lemma \ref{lem:point-sets-lin-independent} in which $G$ consists only of edges $xy$ for which $\textbf{p}_x-\textbf{p}_y \in \{\pm \textbf{u}_1,\dots, \pm \textbf{u}_k\}$ (then each of the forests in the second condition in the lemma is a collection of vertex-disjoint paths). In this case, it is easy to see that the graph $G$ is a subgraph of the infinite graph of all integer sequences in which two sequences are adjacent if and only if they are equal in all coordinates but one, and in this coordinate they differ by one. The precise maximum possible number of edges of a subgraph of $n$ vertices of this graph is known by the isoperimetric inequality of Bollob\'as and Leader \cite[Theorem 15]{BL}. It implies the assertion of the lemma for this special case, and supplies also the tight best possible value for all values of $n$ and not only for powers of two.
\end{rem}

The second ingredient for the proof of \Cref{lem:point-sets-key-lemma,lem:point-sets-lemma-for-distinct-distances} is the following.

\begin{lem}
\label{lem:linear-algebra-greedy}
    Let $V$ be a vector space over a field $F$, and let $\mathbf{u}_1,\dots, \textbf{u}_k\in V$ be non-zero vectors in $V$.  Let $d\ge 1$ and $m\ge 0$ be integers, and suppose that for every subset $I\su \{1,\dots,k\}$, we have $\mathbf{u}_\ell\in \spn_{F}(\mathbf{u}_i\midd i\in I)$ for at most $d \cdot|I|+m$ indices $\ell\in \{1,\dots,k\}$. For some real number $M>0$, let $\lambda_1,\dots,\lambda_k$ be real numbers in the interval $[0,M]$. Then there exists a subset $J\su \{1,\dots,k\}$ such that the vectors $\mathbf{u}_j$ for $j\in J$ are linearly independent (over $F$) and such that \[\sum_{j\in J}\lambda_j\ge \frac{\lambda_1+\dots+\lambda_k-m\cdot M}{d}.\]
\end{lem}

\begin{proof}
    Upon relabelling, we may assume without loss of generality that $\lambda_1\ge \lambda_2\ge \dots\ge \lambda_k$. Let us now construct a sequence of distinct indices $j_1,\dots,j_r\in \{1,\dots,k\}$ as follows. For any $s\ge 1$, having already chosen the indices $j_1,\dots,j_{s-1}$, choose $j_s$ to be the minimum index such that $\mathbf{u}_{j_s}\not\in \spn_{F}(\mathbf{u}_{j_1},\dots, \mathbf{u}_{j_{s-1}})$ if such an index exists. If such an index does not exist, i.e.\ if $\mathbf{u}_{1},\dots,\mathbf{u}_k\in \spn_{F}(\mathbf{u}_{j_1},\dots, \mathbf{u}_{j_{s-1}})$, let us terminate the sequence $j_1,\dots,j_{s-1}$  without choosing any further terms by setting $r=s-1$. Note that when the sequence terminates we have $\mathbf{u}_{1},\dots,\mathbf{u}_k\in \spn_{F}(\mathbf{u}_{j_1},\dots, \mathbf{u}_{j_{r}})$. By our assumption applied to $I=\{j_1,\dots,j_r\}$, this implies that $k\le d \cdot |I| +m= d \cdot r+m$.

    Since  we have $\mathbf{u}_{j_i}\not\in \spn_{F}(\mathbf{u}_{j_1},\dots, \mathbf{u}_{j_{i-1}})$ for $i=1,\dots,m$, the vectors $\mathbf{u}_{j_1},\dots, \mathbf{u}_{j_m}$ are linearly independent.

    Now, we claim that for $s=1,\dots,r$, we have $j_s\le (s-1)d+m+1$. Indeed, suppose that for some $1\le s\le m$, we had $j_s>(s-1)d+m+1$. By the choice of $j_s$, this would mean that $\mathbf{u}_{1},\dots,\mathbf{u}_{(s-1)d+m+1}\in \spn_{F}(\mathbf{u}_{j_1},\dots, \mathbf{u}_{j_{s-1}})$. But this contradicts our assumption for the set $I=\{j_1,\dots,j_{s-1}\}$. Hence, we must indeed have $j_s\le (s-1)d+m+1$ for $s=1,\dots,r$.

    Let us now define $J=\{j_1,\dots,j_r\}$. Then the vectors $\mathbf{u}_j$ for $j\in J$, i.e.\ the vectors $\mathbf{u}_{j_1},\dots, \mathbf{u}_{j_r}$, are linearly independent. Furthermore, defining $0=\lambda_{k+1}=\lambda_{k+2}=\dots$ for notational convenience, we have
    \[d\cdot \sum_{j\in J}\lambda_j= d\cdot \sum_{s=1}^{r}\lambda_{j_s}\ge d\cdot \sum_{s=1}^{r}\lambda_{(s-1)d+m+1}\ge \sum_{s=1}^{r}(\lambda_{(s-1)d+1+m}+\dots+\lambda_{sd+m})=\lambda_{m+1}+\lambda_{m+2}+\dots+\lambda_{rd+m}\]
    Here, we used that $\lambda_1\ge \lambda_2\ge \dots$ and $j_s\le (s-1)d+m+1$ for $s=1,\dots,r$. Recalling $k\le rd+m$  and $0=\lambda_{k+1}=\lambda_{k+2}=\dots$, we can conclude that
    \[d\cdot \sum_{j\in J}\lambda_j\ge \lambda_{m+1}+\lambda_{m+2}+\dots+\lambda_{rd+m}=\lambda_{1}+\dots+\lambda_{k}-(\lambda_1+\dots+\lambda_m)\ge \lambda_{1}+\dots+\lambda_{k}-m\cdot M.\]
    Rearranging now gives $\sum_{j\in J}\lambda_j\ge (\lambda_1+\dots+\lambda_k-mM)/d$, as desired.
\end{proof}

\begin{rem}
The assertion of the \Cref{lem:linear-algebra-greedy} can also be established in a shorter way by applying the matroid partition theorem of Edmonds and Fulkerson \cite{EF}. To do so, define $d+1$ matroids $M_1,\dots,M_{d+1}$ on the set of vectors $\{\textbf{u}_1,\dots,\textbf{u}_k\}$. Each of the first $d$ matroids $M_1,\dots,M_d$ is the usual linear independence matroid: a set of vectors is independent if and only if it is linearly independent. In the last matroid $M_{d+1}$, a set is independent if and only if it contains at most $m$ vectors. If every subset $I$ of the vectors spans at most $d \cdot |I|+m$ of the vectors, then for every set of vectors $J$ the sum of ranks of $J$ according to the $d+1$ matroids $M_1,\dots,M_{d+1}$ is at least $|J|$. By the Edmonds--Fulkerson Theorem this implies that the set $\{\textbf{u}_1,\dots,\textbf{u}_k\}$ can be partitioned into $d+1$  subsets, where one of the subsets is of size at most $m$ and all others are linearly independent sets of vectors. This clearly supplies the conclusion of the lemma.
\end{rem}

Now, let us prove \Cref{lem:point-sets-key-lemma}, 
using \Cref{lem:point-sets-lin-independent,lem:linear-algebra-greedy}.

\begin{proof}[ of \Cref{lem:point-sets-key-lemma}]
    Suppose towards a contradiction that the desired subset $I\su \{1,\dots,k\}$ does not exist. Then for every subset $I\su \{1,\dots,k\}$, we have $\mathbf{u}_\ell\in \spn_{\mathbb{Q}}(\mathbf{u}_i\midd i\in I)$ for at most $d\cdot |I|$ indices $\ell\in \{1,\dots,k\}$. Let $G'$ be the graph with vertex set $\{1,\dots,n\}$ as in the statement of Lemma \ref{lem:point-sets-key-lemma} (where for any $x,y\in \{1,\dots,n\}$ we draw an edge between the vertices $x$ and $y$ if and only if $\mathbf{p}_x-\mathbf{p}_y\in \{\pm \mathbf{u}_1, \dots, \pm \textbf{u}_k\}$). Furthermore, for $i=1,\dots,k$, let $\lambda_i$ be the number of edges $xy$ of $G'$ with $\mathbf{p}_x-\mathbf{p}_y\in \{\pm \mathbf{u}_i\}$. Then $\lambda_1+\dots+\lambda_k\ge e(G')>\frac d2 \cdot n \log n$ by our assumption on the number of edges of $G'$.

    Now, by Lemma \ref{lem:linear-algebra-greedy} applied with $F=\mathbb{Q}$ and $m=0$ there exists a subset $J\su \{1,\dots,k\}$ such that the vectors $\mathbf{u}_j$ for $j\in J$ are linearly independent over $\mathbb{Q}$ and such that $\sum_{j\in J}\lambda_j\ge (\lambda_1+\dots+\lambda_k)/d$. We may assume without loss of generality that $J=\{1,\dots,k'\}$ for some $k'\in \{1,\dots,k\}$ (otherwise, relabel the indices). Then the vectors $\mathbf{u}_1,\dots,\mathbf{u}_{k'}$ are linearly independent and we have $\lambda_1+\dots+\lambda_{k'}\ge (\lambda_1+\dots+\lambda_k)/d>\frac 12 \cdot n \log n$.

    Let us now apply Lemma \ref{lem:point-sets-lin-independent} to $F=\mathbb{Q}$, the linearly independent vectors $\mathbf{u}_1,\dots,\mathbf{u}_{k'}\in V$, the same $\mathbf{p}_1,\dots,\mathbf{p}_n\in V$ as before, and the graph $G$ obtained from $G'$ by only taking the edges $xy$ with $\mathbf{p}_x-\mathbf{p}_y\in \{\pm \mathbf{u}_1, \dots, \pm \textbf{u}_{k'}\}$. Note that this graph $G$ satisfies the two conditions in \Cref{lem:point-sets-lin-independent}. Indeed, for every edge $xy$ of $G$ we have $\mathbf{p}_x-\mathbf{p}_y\in \{\pm \mathbf{u}_i\}\su \spn_{\mathbb{Q}}(\mathbf{u}_i)$ for some $i\in \{1,\dots,k'\}$. Furthermore, for every $i\in \{1,\dots,k'\}$, the edges $xy$ with $\mathbf{p}_x-\mathbf{p}_y\in \spn_{\mathbb{Q}}(\mathbf{u}_i)$ are precisely the edges with $\mathbf{p}_x-\mathbf{p}_y\in \{\pm \mathbf{u}_i\}$ (as $\mathbf{u}_1,\dots,\mathbf{u}_{k'}$ are linearly independent, we cannot have $\mathbf{p}_x-\mathbf{p}_y\in \spn_{\mathbb{Q}}(\mathbf{u}_i)$ if $\mathbf{p}_x-\mathbf{p}_y\in \{\pm \mathbf{u}_j\}$ for $j\neq i$), and these edges form a vertex-disjoint collection of paths and hence a forest. Thus, by Lemma \ref{lem:point-sets-lin-independent} the graph $G$ has at most $\frac 12 \cdot n \log n$ edges. On the other hand, the number of edges of $G$ is $\lambda_1+\dots+\lambda_{k'}>\frac12 \cdot n \log n$, since for every $i=1,\dots,k'$ there are $\lambda_i$ edges $xy$ in $G$ with $\mathbf{p}_x-\mathbf{p}_y\in \{\pm \mathbf{u}_i\}$. This is a contradiction, finishing the proof of Lemma \ref{lem:point-sets-key-lemma}.
\end{proof}

For the proof of Lemma \ref{lem:point-sets-lemma-for-distinct-distances}, we also use the following result, proved by Ungar \cite{ungar} which states that for any $n$ distinct points in $\mathbb{R}^t$, not all on a common line, the pairs of points define at least $n-1$ different line directions.

\begin{thm}[\cite{ungar}]\label{thm-ungar}
Given $n$ distinct points $\mathbf{p}_1,\ldots,\mathbf{p}_n\in \mathbb{R}^t$, not all on one common line, there are at least $n-1$ different one-dimensional linear subspaces of $\mathbb{R}^t$ that are of the form $\spn_{\mathbb{R}}(\mathbf{p}_x-\mathbf{p}_y)$ with $1\le x<y\le n$.
\end{thm}

We remark that Ungar's result is actually slightly stronger, namely in the case of even $n$ he proved that there are at least $n$ (and not only $n-1$) different line directions (and in this stronger version, his bounds are tight). Finally, let us prove Lemma \ref{lem:point-sets-lemma-for-distinct-distances}, using \Cref{lem:point-sets-lin-independent,lem:linear-algebra-greedy} and \Cref{thm-ungar}.

\begin{proof}[ of Lemma \ref{lem:point-sets-lemma-for-distinct-distances}]
Setting $m=\lceil (1-\mu)\cdot n\rceil\le n$, we want to prove that there is a subset $I\su \{1,\dots,k\}$ with $\mathbf{u}_\ell\in \spn_{F}(\mathbf{u}_i\midd i\in I)$ for at least $d\cdot |I|+m+1$ indices $\ell\in \{1,\dots,k\}$. Suppose towards a contradiction that the desired subset $I\su \{1,\dots,k\}$ does not exist. Then for every subset $I\su \{1,\dots,k\}$, we have $\mathbf{u}_\ell\in \spn_{F}(\mathbf{u}_i\midd i\in I)$ for at most $d\cdot |I|+m$ indices $\ell\in \{1,\dots,k\}$.

We may assume that for every $i\in \{1,\dots,k\}$ there exist distinct $x,y\in \{1,\dots,n\}$ with $\mathbf{p}_x-\mathbf{p}_y\in \spn_{F}(\mathbf{u}_i)$ since otherwise, we can just omit all indices $i$ for which this is not the case, and relabel the remaining indices.

By Theorem \ref{thm-ungar}, there are at least $n-1$ different line directions in $\spn_\R(\mathbf{p}_1,\dots,\mathbf{p}_n)\su V$ appearing among the differences $\mathbf{p}_x-\mathbf{p}_y$ with $1\le x<y\le n$. For each of these differences we have $\mathbf{p}_x-\mathbf{p}_y\in \spn_F(\mathbf{u}_i)$ for some $i\in \{1,\dots,k\}$.
Hence, there must be at least $n-1$ different vectors $\mathbf{u}_i$, so $k\ge n-1$.

Let us now construct a sequence of distinct indices $j_1,\dots,j_r\in \{1,\dots,k\}$ recursively as follows. Let $s\ge 1$, and assume that we have already chosen the indices $j_1,\dots,j_{s-1}$. Let us now look at the image of the point set $\{\mathbf{p}_1,\dots,\mathbf{p}_n\}\su V$ under the projection map $V\to V/\spn_F(\mathbf{u}_{j_1},\dots,\mathbf{u}_{j_{s-1}})$. For every point in this image, let us pick a preimage $\mathbf{p}_h$ and let $H_{s-1}\su \{1,\dots,n\}$ be the resulting set of indices $h\in \{1,\dots,n\}$ for these preimages. Then $|H_{s-1}|$ is precisely the size of the image of $\{\mathbf{p}_1,\dots,\mathbf{p}_n\}$ in $V/\spn_F(\mathbf{u}_{j_1},\dots,\mathbf{u}_{j_{s-1}})$ and for every $x\in \{1,\dots,n\}$ there is exactly one $h\in H_{s-1}$ with $\mathbf{p}_x-\mathbf{p}_h\in \spn_F(\mathbf{u}_{j_1},\dots,\mathbf{u}_{j_{s-1}})$. If there exists an index $i\in \{1,\dots,k\}$ for which there is a forest on the vertex set $H_{s-1}$ with at least $\frac{3d}{\mu}$ edges such that for every edge $xy$ of the forest we have $\mathbf{p}_x-\mathbf{p}_y\in \spn_F(\mathbf{u}_i)$, then let us pick such an index $i$ and define $j_s=i$. If there is no such index $i\in \{1,\dots,k\}$, let us terminate the sequence $j_1,\dots,j_{s-1}$ without choosing any further terms by setting $r=s-1$.

Recall that for $s=0,\dots,r$, we defined $H_s\su \{1,\dots,n\}$ such that $|H_s|$ is the size of the image of the point set $\{\mathbf{p}_1,\dots,\mathbf{p}_n\}\su V$ under the projection map $V\to V/\spn_F(\mathbf{u}_{j_1},\dots,\mathbf{u}_{j_{s}})$, and such that for every $x\in \{1,\dots,n\}$ there is exactly one $h\in H_{s}$ with $\mathbf{p}_x-\mathbf{p}_h\in \spn_F(\mathbf{u}_{j_1},\dots,\mathbf{u}_{j_{s}})$. In particular, we have $n=|H_0|\ge |H_1|\ge \dots\ge |H_r|\ge 1$. The following claim gives an upper bound on the length of our sequence $j_1,\dots,j_r$.

\begin{claim}\label{claim-bound-on-r} We have $r\le \frac{\mu}{3d}\cdot n$.
\end{claim}
\begin{cla_proof}
In order to show the claim, it suffices to prove that $|H_s|\le |H_{s-1}|-\frac{3d}{\mu}$ for $s=1,\dots,r$. Indeed, then we have $1\le |H_r|\le |H_0|-r\cdot \frac{3d}{\mu}=n-r\cdot \frac{3d}{\mu}$, which implies $r\le \frac{\mu}{3d}\cdot n$.

So let $s\in \{1,\dots,r\}$, and consider the set $H_{s-1}\su \{1,\dots,n\}$. The image of the point set $\{\mathbf{p}_1,\dots,\mathbf{p}_n\}$ in the quotient space $V'=V/\spn_F(\mathbf{u}_{j_1},\dots,\mathbf{u}_{j_{s-1}})$ is the same as the image of the point set $\{\mathbf{p}_h\midd h\in H_{s-1}\}$ in this quotient space $V'$.

Now, $|H_s|$ is the size of this image under the additional projection given by $V'\to V'/\spn_F(\mathbf{u}_{j_s}) \linebreak= V/\spn_F(\mathbf{u}_{j_1},\dots,\mathbf{u}_{j_{s}})$. Recall that by the definition of $j_s$, there is a forest on the vertex set $H_{s-1}$ with at least $3d/\mu$ edges such that for every edge $xy$ of the forest we have $\mathbf{p}_x-\mathbf{p}_y\in \spn_F(\mathbf{u}_{j_s})$. For every such edge $xy$, the points $\mathbf{p}_x$ and $\mathbf{p}_y$ have the same image in $V'/\spn_F(\mathbf{u}_{j_s})$. So for every connected component of this forest, the corresponding points $\mathbf{p}_x$ are mapped to the same point in $V'/\spn_F(\mathbf{u}_{j_s})$. Since the forest has at least $\frac{3d}{\mu}$ edges, it has at most $|H_{s-1}|-\frac{3d}{\mu}$ connected components. Hence, the image of the point set $\{\mathbf{p}_h\midd h\in H_{s-1}\}$ in $V'/\spn_F(\mathbf{u}_{j_s})$ has size at most $|H_{s-1}|-\frac{3d}{\mu}$, meaning that indeed $|H_s|\le |H_{s-1}|-\frac{3d}{\mu}$.
\end{cla_proof}

\begin{claim}\label{claim:bound-on-H-r} We have $|H_r|> \frac{\mu}{3d}\cdot n$.
\end{claim}
\begin{cla_proof}
Suppose towards a contradiction that $|H_r|\le \frac{\mu}{3d}\cdot n$. Let us fix some index $h^*\in H_r$. Then for every $h\in H_r$ by the assumption in Lemma \ref{lem:point-sets-lemma-for-distinct-distances} we can choose an index $i(h)\in \{1,\dots,k\}$ with $\mathbf{p}_h-\mathbf{p}_{h^*}\in\spn_F(\mathbf{u}_{i(h)})$. Now, defining $I=\{j_1,\dots,j_r\}\cup \{i(h)\midd h\in H_r\}$, we have $|I|\le r+|H_r|\le 2\cdot \frac{\mu}{3d}\cdot n$ using Claim \ref{claim-bound-on-r}.

Note that now we have $\mathbf{p}_h-\mathbf{p}_{h^*}\in \spn_{F}(\mathbf{u}_i\midd i\in I)$ for all $h\in H_r$. Furthermore, for every $x\in \{1,\dots,n\}$ there is some $h\in H_{r}$ with $\mathbf{p}_x-\mathbf{p}_h\in \spn_F(\mathbf{u}_{j_1},\dots,\mathbf{u}_{j_{r}})\su \spn_{F}(\mathbf{u}_i\midd i\in I)$. This implies that $\mathbf{p}_x-\mathbf{p}_{h^*}=\mathbf{p}_x-\mathbf{p}_{h}+\mathbf{p}_h-\mathbf{p}_{h^*}\in \spn_{F}(\mathbf{u}_i\midd i\in I)$ for all $x\in \{1,\dots,n\}$. Thus, we obtain $\mathbf{p}_x-\mathbf{p}_{y}=\mathbf{p}_x-\mathbf{p}_{h^*}-(\mathbf{p}_y-\mathbf{p}_{h^*})\in \spn_{F}(\mathbf{u}_i\midd i\in I)$ for all $x,y\in \{1,\dots,n\}.$

Now, we claim that we have $\mathbf{u}_\ell\in \spn_{F}(\mathbf{u}_i\midd i\in I)$ for all $\ell\in \{1,\dots,k\}$. Indeed, for each $\ell\in \{1,\dots,k\}$, we assumed that there exist two distinct indices $x,y\in \{1,\dots,k\}$ with $\mathbf{p}_x-\mathbf{p}_y\in \spn_{F}(\mathbf{u}_\ell)$. Since $\mathbf{p}_x\ne \mathbf{p}_y$, this means that $\mathbf{p}_x-\mathbf{p}_y=t\mathbf{u}_\ell$ for some $t\in F\setminus \{0\}$. But now we have $t\mathbf{u}_\ell=\mathbf{p}_x-\mathbf{p}_y\in \spn_{F}(\mathbf{u}_i\midd i\in I)$, implying that $\mathbf{u}_\ell\in \spn_{F}(\mathbf{u}_i\midd i\in I)$ as claimed.

Thus, the number of indices $\ell\in \{1,\dots,k\}$ with $\mathbf{u}_\ell\in \spn_{F}(\mathbf{u}_i\midd i\in I)$ is
\[k\ge n-1\ge \frac{2\mu}{3}\cdot n+\lceil (1-\mu)\cdot n\rceil+1 = d\cdot \frac{2\mu}{3d}\cdot n+m+1\ge d\cdot |I|+m+1,\]
if $n$ is sufficiently large with respect to $\mu$. This contradicts our assumption that such a set $I$ does not exist.
\end{cla_proof}

Now, for each $i=1,\dots,k$, let us choose a forest $G_i$ on the vertex set $H_r$ with the maximum possible number of edges such that for every edge $xy$ of the forest we have $\mathbf{p}_x-\mathbf{p}_y\in \spn_F(\mathbf{u}_{i})$. Let $\lambda_i$ be the number of edges of this forest, and note that $\lambda_i\le \frac{3d}{\mu}$ by our construction of the sequence $j_1,\dots,j_r$ (more precisely, by the fact that the sequence terminates at $j_r$). In particular, the forest $G_i$ has at most $2\lambda_i$ non-isolated vertices.

Recall that for any pair $(x,y)\in H_r \times H_r$ with $x\ne y$, we have $\mathbf{p}_x-\mathbf{p}_y\in \spn_F(\mathbf{u}_i)$ for some $i\in \{1,\dots,k\}$. We claim that for each $i\in \{1,\dots,k\}$ there are at most $(2\lambda_i)^2$ pairs $(x,y)\in H_r \times H_r$ with $x\ne y$ such that $\mathbf{p}_x-\mathbf{p}_y\in \spn_F(\mathbf{u}_i)$. Indeed, for any such pair $(x,y)$ the vertices $x$ and $y$ must lie in the same component of the forest $G_i$ (since otherwise we could add the edge $xy$ to the forest $G_i$, which would be a contradiction to our choice of the forest $G_i$). Since $x\ne y$, this means that $x$ and $y$ are both non-isolated vertices in $G_i$, meaning that there are at most $2\lambda_i$ choices for $x$ and at most $2\lambda_i$ choices for $y$. So for each $i\in \{1,\dots,k\}$, there are indeed at most $(2\lambda_i)^2$ pairs $(x,y)\in H_r \times H_r$ with $x\ne y$ satisfying $\mathbf{p}_x-\mathbf{p}_y\in \spn_F(\mathbf{u}_i)$.

Since for each of the $|H_r|\cdot (|H_r|-1)$ pairs $(x,y)\in H_r \times H_r$ with $x\ne y$ there exists an index $i\in \{1,\dots,k\}$ with $\mathbf{p}_x-\mathbf{p}_y\in \spn_F(\mathbf{u}_i)$, we can conclude that
\[\frac{|H_r|^2}{2}\le |H_r|\cdot (|H_r|-1)\le \sum_{i=1}^{k}(2\lambda_i)^2= 4\cdot \sum_{i=1}^{k}\lambda_i^2\le \frac{12d}{\mu}\sum_{i=1}^{k}\lambda_i.\]
Here we used that $|H_r|> \frac{\mu}{3d}\cdot n\ge 2$, which holds by Claim \ref{claim:bound-on-H-r}, and that $\lambda_i\le \frac{3d}{\mu}$ for $i=1,\dots,k$. Rearranging yields $\lambda_1+\dots+\lambda_k\ge \frac{\mu}{24d}\cdot |H_r|^2$.

Let us now apply Lemma \ref{lem:linear-algebra-greedy} with $M=\frac{3d}{\mu}$, recalling our assumption that for every subset $I\su \{1,\dots,k\}$, we have $\mathbf{u}_\ell\in \spn_{F}(\mathbf{u}_i\midd i\in I)$ for at most $d\cdot |I|+m$ indices $\ell\in \{1,\dots,k\}$. By Lemma \ref{lem:linear-algebra-greedy}, there is a subset $J\su \{1,\dots,k\}$ such that the vectors $\mathbf{u}_j$ for $j\in J$ are linearly independent over $F$ and such that
\[\sum_{j\in J}\lambda_j\ge \frac{\lambda_1+\dots+\lambda_k-m\cdot \frac{3d}{\mu}}{d}\ge \frac{\frac{\mu}{24d}\cdot |H_r|^2-n\cdot \frac{3d}{\mu}}{d}.\]
By Claim \ref{claim:bound-on-H-r} we have $n\cdot \frac{3d}{\mu}\le \left(\frac{3d}{\mu}\right)^2\cdot |H_r|\le \frac{\mu}{48d}\cdot |H_r|^2$ if $n$ (and therefore also $|H_r|\ge \frac{\mu}{3d}\cdot n$) is sufficiently large 
with respect to $d$ and $\mu$. So we can conclude that
\[\sum_{j\in J}\lambda_j\ge\frac{\frac{\mu}{24d}\cdot |H_r|^2-n\cdot \frac{3d}{\mu}}{d}\ge \frac{\frac{\mu}{48d}\cdot |H_r|^2}{d}=\frac{\mu}{48d^2}\cdot |H_r|^2.\]
Finally, let $G$ be the graph obtained by taking all edges of the forests $G_j$ for $j\in J$. As the vectors $\mathbf{u}_j$ for $j\in J$ are linearly independent over $F$, the edge sets of these forests are pairwise disjoint, and hence $G$ has $\sum_{j\in J}\lambda_j\ge \frac{\mu}{48d^2} \cdot |H_r|^2$ edges. On the other hand, $G$ satisfies the assumptions in Lemma \ref{lem:point-sets-lin-independent}, and so Lemma \ref{lem:point-sets-lin-independent} implies that $G$ has at most $\frac12 \cdot |H_r|\log |H_r|$ edges. Thus,
\[\frac{\mu}{48d^2}\cdot |H_r|^2\le e(G)\le \frac12 \cdot |H_r|\log |H_r|,\]
which is a contradiction if $n$ (and hence also $|H_r|\ge \frac{\mu}{3d}\cdot n$) is sufficiently large 
with respect to $d$ and $\mu$.
\end{proof}

\section{Proofs of Theorems \ref{thm:main} and \ref{thm:distinct-distances}: most norms have few special distances}
\label{S4}

Fix $d\ge 2$. Recall that $\B_d$ is the collection of all closed, bounded, $\textbf{0}$-symmetric convex bodies in $\R^d$ with $\textbf{0}$ in the interior. The sets $B\in \B_d$ are in direct correspondence with the norms $\nn$ on $\R^d$ (where to every norm $\nn$ on $\R^d$ we associate the set $B\in \B_d$ arising as the unit ball of $\nn$).

Let $\A \subseteq \B_d$ be the set of all $B\in \B_d$ arising as the unit ball of a norm $\nn$ on $\R^d$ such that for some $n\ge 1$ there exist $n$ points in $\R^d$ with more than $\frac d2\cdot n \log n$ unit distances according to the norm $\nn$. In order to prove Theorem \ref{thm:main} we need to show that $\A \subseteq \B_d$ is a meagre set.

For Theorem \ref{thm:distinct-distances}, let $\mu>0$ and let $n_0(d,\mu)>1/\mu$ be such that the statement in Lemma \ref{lem:point-sets-lemma-for-distinct-distances} holds for all $n\ge n_0(d,\mu)$. Now, let $\A^*_\mu \subseteq \B_d$ be the set of all $B\in \B_d$ arising as the unit ball of a norm $\nn$ on $\R^d$ such that for some $n\ge n_0(d,\mu)$ there exist $n$ points in $\R^d$ with at most $(1-\mu)\cdot n$ distinct distances according to the norm $\nn$. In order to prove Theorem \ref{thm:distinct-distances} it suffices to show that $\A^*_\mu \subseteq \B_d$ is a meagre set. Indeed, then for most norms on $\mathbb{R}^d$ it is true that for all $n\ge n_0(d,\mu)$ any set of $n$ points in $\mathbb{R}^d$ has at least $(1-\mu)\cdot n$ distinct distances appearing. This means that the number of distinct distances is of the form $(1-o(1))\cdot n$.

Let us define $m=0$ in the setting of Theorem \ref{thm:main} and $m=\lceil (1-\mu)n \rceil$ in the setting of Theorem \ref{thm:distinct-distances}.

To show  that $\A$ and $\A^*_\mu$ are meagre sets, we need to show that $\A$ and $\A^*_\mu$ can be covered by a countable union of nowhere dense subsets of $\B_d$. We will consider suitably defined subsets $\A_{A,\eta}\su \B_d$, indexed by a $(d\ell+m+1) \times \ell$ matrix $A$ (for a positive integer $\ell$) and some rational number $0<\eta<\pi/2$. In the setting of Theorem \ref{thm:main}, i.e.\ to cover the set $\A$, we will consider rational matrices $A\in \mathbb{Q}^{(d\ell+m+1) \times \ell}$. In the setting of Theorem \ref{thm:distinct-distances}, i.e.\ to cover the set $\A^*_\mu$, our matrices $A$ will have entries in the function field $\mathbb{Q}(x_1,\dots,x_m)$ (i.e.\ in the field of rational functions in $m$ variables with rational coefficients). In either case, the entries of $A$ are chosen from a countable field, and so there are only countably many choices for such $A$ and $\eta$. Thus, it suffices to show that $\A\su \bigcup_{A, \eta}\A_{A,\eta}$ or $\A^*_\mu\su \bigcup_{A, \eta}\A_{A,\eta}$, respectively, and that  each of the sets $\A_{A,\eta}\su \B_d$ is nowhere dense in $\B_d$.

Let us now define our sets $\A_{A,\eta}\su \B_d$, starting with the setting of Theorem \ref{thm:main} (recall that then $m=0$). Given a rational $(d\ell+1) \times \ell$ matrix $A=(a_{ji})$ and given a rational $0<\eta<\pi/2$, let $\A_{A,\eta}\su \B_d$ consist of the unit balls $B\in \B_d$ of all norms $\nn$ on $\R^d$ for which there exist unit vectors $\textbf{u}_1,\ldots, \textbf{u}_{d\ell+1}\in \R^d$ (i.e.\ vectors with endpoints on the boundary of $B$) satisfying the following two conditions:
\begin{compactitem}
    \item $\textbf{u}_{j}=\sum_{i=1}^{\ell} a_{ji}\textbf{u}_i$ for $j=1,\dots,d\ell+1$.
    \item For all $1\le i<j\le d\ell+1$, the angle between the two lines $\spn_{\R}(\textbf{u}_i)$ and $\spn_{\R}(\textbf{u}_j)$ is larger than $\eta$.
\end{compactitem}

In the setting of Theorem \ref{thm:distinct-distances}, given a $(d\ell+m+1) \times \ell$ matrix $A=(a_{ji})$ with entries in $\mathbb{Q}(x_1,\dots,x_m)$ and given a rational $0<\eta<\pi/2$, let $\A_{A,\eta}\su \B_d$ consist of the unit balls $B\in \B_d$ of all norms $\nn$ on $\R^d$ for which there exist $z_1,\dots,z_m\in \mathbb{R}$ and unit vectors $\textbf{u}_1,\ldots, \textbf{u}_{d\ell+m+1}\in \R^d$ (i.e.\ vectors with endpoints on the boundary of $B$) satisfying the following three conditions:
\begin{compactitem}
    \item For every entry $a_{ji}\in \mathbb{Q}(x_1,\dots,x_m)$ of the matrix $A$, the evaluation $a_{ji}(z_1,\dots,z_m)\in \mathbb{R}$ is well-defined (i.e.\ the polynomial in the denominator of $a_{ji}$ does not evaluate to zero at $(z_1,\dots,z_m)$).
    \item $\textbf{u}_{j}=\sum_{i=1}^{\ell} a_{ji}(z_1,\dots,z_m)\textbf{u}_i$ for $j=1,\dots,d\ell+m+1$.
    \item For all $1\le i<j\le d\ell+m+1$, the angle between the two lines $\spn_{\R}(\textbf{u}_i)$ and $\spn_{\R}(\textbf{u}_j)$ is larger than $\eta$.
\end{compactitem}

In order to prove \Cref{thm:main,thm:distinct-distances}, it suffices to show that each of the sets $\A_{A,\eta}$ is nowhere dense and that $\A\su \bigcup_{A, \eta}\A_{A,\eta}$ and $\A^*_\mu\su \bigcup_{A, \eta}\A_{A,\eta}$, respectively. These statements will be the content of the following lemmas. 

\begin{lem}\label{lem:conatined-in-union-A-A-eta}
In the setting of Theorem \ref{thm:main}, we have $\A\su \bigcup_{A,\eta}\A_{A,\eta}$, where the union is over all rational $(d\ell+1) \times \ell$ matrices $A\in \mathbb{Q}^{(d\ell+1) \times \ell}$ for all positive integers $\ell$, and all rational numbers $0<\eta<\pi/2$.
\end{lem}

\begin{proof}
Let $B\in \A$. Then $B$ is the unit ball of a norm $\nn$ on $\R^d$ with the property that there exist distinct points $\mathbf{p}_1,\dots,\mathbf{p}_n\in \R^d$ such that there are more than $\frac d2\cdot n \log n$ unit distances according to the norm $\nn$ between the points $\mathbf{p}_1,\dots,\mathbf{p}_n$. 
Now, let $\textbf{u}_1,\ldots,\textbf{u}_k\in \R^d$ with $||\textbf{u}_1||=\dots=||\textbf{u}_k||=1$ be the unit vectors (signed arbitrarily) corresponding to the unit distances between the points $\mathbf{p}_1,\dots,\mathbf{p}_n$. Note that $\textbf{u}_1,\ldots,\textbf{u}_k$ are non-zero vectors in distinct directions (i.e.\ the lines $\spn_{\R}(\textbf{u}_i)$ for $i=1,\dots,k$ are all distinct).

Let us now consider $\mathbb{R}^d$ as an (infinite-dimensional) vector space over $\mathbb{Q}$, and let us apply \Cref{lem:point-sets-key-lemma}. As in the lemma statement, consider the graph with vertex set $\{1,\dots,n\}$, where for any $x,y\in \{1,\dots,n\}$ we draw an edge between the vertices $x$ and $y$ if and only if $\mathbf{p}_x-\mathbf{p}_y\in \{\pm \mathbf{u}_1, \dots, \pm \textbf{u}_k\}$. Then the edges correspond precisely to the unit distances according to $\nn$ among the points $\mathbf{p}_1,\dots,\mathbf{p}_n$, and so the graph has more than $\frac d2 \cdot  n \log n$ edges. Thus, by Lemma \ref{lem:point-sets-key-lemma}, there exists a subset $I\su \{1,\dots,k\}$, such that we have $\mathbf{u}_j\in \spn_{\mathbb{Q}}(\mathbf{u}_i\midd i\in I)$ for at least $d \cdot |I|+1$ indices $j\in \{1,\dots,k\}$. Note that we must have $I\neq \emptyset$, since $\textbf{u}_1,\ldots,\textbf{u}_k$ are non-zero vectors.

Upon relabelling the indices, we may assume that $I=\{1,\dots,\ell\}$ for some integer $\ell\ge 1$, and that $\mathbf{u}_1,\dots, \mathbf{u}_{d\ell+1}\in \spn_{\mathbb{Q}}(\mathbf{u}_i\midd i\in I)=\spn_{\mathbb{Q}}(\mathbf{u}_1,\dots, \mathbf{u}_{\ell})$. Then we can find coefficients $a_{ji}\in \mathbb{Q}$ for $j=1,\dots,d\ell+1$ and $i=1,\dots,\ell$ such that $\mathbf{u}_j=\sum_{i=1}^{\ell}a_{ji}\mathbf{u}_i$ for all $j=1,\dots,d\ell+1$.

We can now take $A$ to be the $(d\ell+1)\times \ell$ matrix with entries $a_{ji}$ for $j=1,\dots,d\ell+1$ and $i=1,\dots,\ell$. 
Furthermore, consider the $d\ell+1$ lines $\spn_{\R}(\textbf{u}_i)$ for $i=1,\dots,d\ell+1$ and choose $0<\eta<\pi/2$ to be rational and smaller than the angle between any two of these lines. 
Then $\textbf{u}_1,\ldots, \textbf{u}_{d\ell+1}$ are unit vectors according to the norm $\nn$, such that $\mathbf{u}_j=\sum_{i=1}^{\ell}a_{ji}\mathbf{u}_i$ for all $j=1,\dots,d\ell+1$ and such that for all $1\le i<j\le d\ell+1$ the angle between the lines $\spn_{\R}(\textbf{u}_i)$ and $\spn_{\R}(\textbf{u}_j)$ is larger than $\eta$. This means that the unit ball $B$ of the norm $\nn$ belongs to $\A_{A,\eta}$, as desired. 
\end{proof}

\begin{lem}
In the setting of Theorem \ref{thm:distinct-distances}, we have $\A^*_\mu\su \bigcup_{A,\eta}\A_{A,\eta}$, where the union is over all $(d\ell+m+1) \times \ell$ matrices $A\in \mathbb{Q}(x_1,\dots,x_m)^{(d\ell+m+1) \times \ell}$ for all positive integers $\ell$, and all rational numbers $0<\eta<\pi/2$.
\end{lem}

\begin{proof}
Let $B\in \A^*_\mu$. Then $B$ is the unit ball of a norm $\nn$ on $\R^d$ with the property that for some $n\ge n_0(d,\mu)$ there exist distinct points $\mathbf{p}_1,\dots,\mathbf{p}_n\in \R^d$ such that there are at most $(1-\mu)\cdot n\le m$ distinct distances according to the norm $\nn$ between the points $\mathbf{p}_1,\dots,\mathbf{p}_n$. 
Now, let $\textbf{u}_1,\ldots,\textbf{u}_k\in \R^d$ with $||\textbf{u}_1||=\dots=||\textbf{u}_k||=1$ be the unit vectors (signed arbitrarily) in the directions of all the differences between the points $\mathbf{p}_1,\dots,\mathbf{p}_n$, and let $z_1,\dots,z_m>0$ be positive real numbers such that $||\mathbf{p}_x-\mathbf{p}_y||\in \{z_1,\dots,z_m\}$ for all distinct $x,y\in \{1,\dots,n\}$. Note that then for all distinct $x,y\in \{1,\dots,n\}$ we have $\mathbf{p}_x-\mathbf{p}_y = \pm z_j\textbf{u}_i$ for some $j\in \{1,\dots,m\}$ and $i\in \{1,\dots,k\}$. Defining $F=\mathbb{Q}(z_1,\dots,z_m)\su \mathbb{R}$ to be the field extension of $\mathbb{Q}$ generated by $z_1,\dots,z_m$, we obtain that for all distinct $x,y\in \{1,\dots,n\}$ we have $\mathbf{p}_x-\mathbf{p}_y \in \spn_F(\textbf{u}_i)$ for some $i\in \{1,\dots,k\}$. Furthermore, note that $\textbf{u}_1,\ldots,\textbf{u}_k$ are non-zero vectors in distinct directions (i.e.\ the lines through $\textbf{u}_1,\ldots,\textbf{u}_k$ are distinct).

If all of the points $\mathbf{p}_1,\dots,\mathbf{p}_n$ are on a common (affine) line in $\mathbb{R}^d$, then there are at least $n-1$ distinct distances among these points according to the norm $\nn$. However we have $n-1>(1-\mu)\cdot n$ (since $n\ge n_0(d,\mu)>1/\mu$), so this would be a contradiction. Hence, the points $\mathbf{p}_1,\dots,\mathbf{p}_n$ do not all lie on a common line in $\mathbb{R}^d$.

By Lemma \ref{lem:point-sets-lemma-for-distinct-distances} applied to $V=\mathbb{R}^d$ and $F=\mathbb{Q}(z_1,\dots,z_m)\su \mathbb{R}$, there exists a subset $I\su \{1,\dots,k\}$, such that we have $\mathbf{u}_j\in \spn_{F}(\mathbf{u}_i\midd i\in I)$ for at least $d \cdot |I|+(1-\mu)\cdot n+1$ indices $j\in \{1,\dots,k\}$. As the number of such indices is an integer, and $m=\lceil (1-\mu)n\rceil$, we have $\mathbf{u}_j\in \spn_{F}(\mathbf{u}_i\midd i\in I)$ for at least $d \cdot|I|+m+1$ indices $j\in \{1,\dots,k\}$. Note that we must have $I\neq \emptyset$, since $\textbf{u}_1,\ldots,\textbf{u}_k$ are non-zero vectors.

Upon relabelling the indices if necessary, we may assume that $I=\{1,\dots,\ell\}$ for some integer $\ell\ge 1$, and that $\mathbf{u}_1,\dots, \mathbf{u}_{d\ell+m+1}\in \spn_{F}(\mathbf{u}_i\midd i\in I)=\spn_{F}(\mathbf{u}_1,\dots, \mathbf{u}_{\ell})$. Then we can find coefficients $a_{ji}^*\in F=\mathbb{Q}(z_1,\dots,z_m)$ for $j=1,\dots,d\ell+m+1$ and $i=1,\dots,\ell$ such that $\mathbf{u}_j=\sum_{i=1}^{\ell}a_{ji}^*\mathbf{u}_i$ for all $j=1,\dots,d\ell+m+1$. For each of these coefficients $a_{ji}^*\in \mathbb{Q}(z_1,\dots,z_m)$, we can choose a rational function $a_{ji}\in \mathbb{Q}(x_1,\dots,x_m)$ such that $a_{ji}(z_1,\dots,z_m)=a_{ji}^*$ (i.e.\ $a_{ji}$ evaluates to $a_{ji}^*$ when plugging in $z_1,\dots,z_m$ for the abstract variables $x_1,\dots,x_m$ in the function field $\mathbb{Q}(x_1,\dots,x_m)$). Note that in particular, all the evaluations $a_{ji}(z_1,\dots,z_m)$ for $j=1,\dots,d\ell+m+1$ and $i=1,\dots,\ell$ are well-defined.

We can now take $A$ to be the $(d\ell+m+1)\times \ell$ matrix with entries $a_{ji}$ for $j=1,\dots,d\ell+m+1$ and $i=1,\dots,\ell$. 
Furthermore, consider the $d\ell+m+1$ lines $\spn_{\R}(\textbf{u}_i)$ for $i=1,\dots,d\ell+m+1$ and choose $0<\eta<\pi/2$ to be rational and smaller than the angle between any two of these lines. 
Then $\textbf{u}_1,\ldots, \textbf{u}_{d\ell+m+1}$ are unit vectors according to the norm $\nn$, such that $\mathbf{u}_j=\sum_{i=1}^{\ell}a^*_{ji}\mathbf{u}_i=\sum_{i=1}^{\ell}a_{ji}(z_1,\dots,z_m)\mathbf{u}_i$ for all $j=1,\dots,d\ell+m+1$ and such that for all $1\le i<j\le d\ell+m+1$ the angle between the lines $\spn_{\R}(\textbf{u}_i)$ and $\spn_{\R}(\textbf{u}_j)$ is larger than $\eta$. This means that the unit ball $B$ of the norm $\nn$ belongs to $\A_{A,\eta}$, as desired. 
\end{proof}

It remains to show that each of the sets $\A_{A,\eta}$ is nowhere dense in $\B_d$. This is the content of the following lemma.

\begin{lem}\label{lem:A-A-eta-nowhere-dense}
Let $\ell>0$ be an integer, let $0<\eta<\pi/2$ be a rational number, and let $A$ be a $(d\ell+m+1)\times \ell$ matrix with entries in $\mathbb{Q}$ or in $\mathbb{Q}(x_1,\dots,x_m)$. Then $\A_{A,\eta}$ is nowhere dense in $\B_d$.
\end{lem}

\begin{proof}
As in the lemma statement, let the matrix $A=(a_{ji})$ and $0<\eta<\pi/2$ be fixed.  To show that $\A_{A,\eta}$ is nowhere dense in $\B_d$, we need to show that for every $B_0\in \B_d$ and every $\eps>0$ there exist $B \in \B_d$ and $\delta>0$ with $d_H(B,B_0)<\eps$ such that no $B' \in \B_d$ with $d_H(B,B')<\delta$ belongs to $\A_{A,\eta}$. So let us fix $B_0\in \B_d$. It suffices to show the desired statement for small $\eps>0$ (note that if the statement holds for some $\eps>0$, then it holds for any larger $\eps$ as well).

Recall that $B_0\su \R^d$ is a closed, bounded, $\textbf{0}$-symmetric convex body containing $\textbf{0}$ in its interior. Due to the latter condition, there exists some $s>0$ such that $||\textbf{b}||_2\ge 2s$ for all $\textbf{b}$ on the boundary of $B_0$. By making $\eps$ smaller if needed, we may assume that $\eps<s$.

By Lemma \ref{lem:approximating-polytope}, applied with $\mu=\min\{\eps/2,s \cdot \sin (\eta/3)\}$, we can approximate $B_0$ with a bounded $\textbf{0}$-symmetric polytope $B_1$ containing $\textbf{0}$ in its interior such that $d_H(B_0,B_1)<\eps/2$ and all facets of $B_1$ have diameter at most $s\cdot \sin (\eta/3)$ (with respect to the Euclidean distance). By Lemma \ref{lem:Hausdorff-distance-boundary} every point $\textbf{b}$ on the boundary of $B_1$ has Euclidean distance at most $\eps/2<s$ to some point on the boundary of $B_0$. In particular, we have $||\textbf{b}||_2\ge s$ for all $\textbf{b}$ on the boundary of $B_1$ (so $B_1$ contains the Euclidean ball with radius $s$ around the origin).

Let us now choose a system of linear inequalities of the form $|\textbf{o}_i\cdot \textbf{x}|\le s_i$ for $i=1,\dots,h$ describing the polytope $B_1$ (with non-zero vectors $\textbf{o}_1,\dots,\textbf{o}_h\in \R^d$ and real numbers $s_1,\dots,s_h>0$). More formally, $B_1$ is the set of all $\textbf{x}\in \R^d$ such that $|\textbf{o}_i\cdot \textbf{x}|\le s_i$ for $i=1,\dots,h$. By rescaling the inequalities, we may choose the vectors $\textbf{o}_1,\dots,\textbf{o}_h\in \R^d$ such that $||\textbf{o}_1||_2=\dots=||\textbf{o}_h||_2=1$. As $B_1$ contains the Euclidean ball of radius $s$ around the origin, we have $s_i\ge s$ for $i=1,\dots,h$.

Note that geometrically, one can think of $\textbf{o}_1,\dots,\textbf{o}_h$ as the (Euclidean) unit vectors orthogonal to the (parallel pairs of) facets of $B_1$. Furthermore, note that every point $\textbf{b}$ on the boundary of $B_1$ satisfies $|\textbf{o}_i\cdot \textbf{b}|= s_i$ for some $i\in \{1,\dots,h\}$.

In the setting of Theorem \ref{thm:main} (where $m=0$ and $A$ has entries in $\mathbb{Q}$), let us say that an $h$-tuple $(t_1,\dots,t_h)\in \R^h$ is \emph{achievable} if there exist distinct indices $\varphi(1),\dots,\varphi(d\ell+1)\in \{1,\dots,h\}$ and vectors $\textbf{u}_1,\ldots, \textbf{u}_{d\ell+1}\in \R^d$ satisfying the following two conditions:
\begin{compactitem}
    \item[(a)] $\textbf{u}_{j}=\sum_{i=1}^{\ell} a_{ji}\textbf{u}_i$ for $j=1,\dots,d\ell+1$, and
    \item[(b)] $|\textbf{o}_{\varphi(j)} \cdot \textbf{u}_j|=t_{\varphi(j)}$ for $j=1,\dots,d\ell+1$.
\end{compactitem}
Intuitively, an $h$-tuple $(t_1,\dots,t_h)$ is achievable if there is a solution $(\textbf{u}_1,\ldots, \textbf{u}_{d\ell+1})$ to the linear equations prescribed by the matrix $A$, such that when plugging in the $d\ell+1$ vectors $\textbf{u}_1,\ldots, \textbf{u}_{d\ell+1}$ into some choice of $d\ell+1$ different expressions $|\textbf{o}_{i} \cdot \textbf{x}|$ for $i\in\{1,\dots,h\}$, one obtains the corresponding entries of $(t_1,\dots,t_h)$. This will be relevant in our arguments in the following way: We will consider the polytope $\{\textbf{x}\in \R^d \midd |\textbf{o}_i\cdot \textbf{x}|\le t_i\text{ for }i=1,\dots,h\}$ (note that for $(t_1,\dots,t_h)$ close to $(s_1,\dots,s_h)$, this can be viewed as a perturbation of the polytope $B_1=\{\textbf{x}\in \R^d \midd |\textbf{o}_i\cdot \textbf{x}|\le s_i\text{ for }i=1,\dots,h\}$). Whenever there is a solution $(\textbf{u}_1,\ldots, \textbf{u}_{d\ell+1})$ to the linear equations prescribed by $A$ such that $\textbf{u}_1,\ldots, \textbf{u}_{d\ell+1}$ lie on distinct facets of this polytope, then the $h$-tuple $(t_1,\dots,t_h)$ must be achievable (indeed, each $\textbf{u}_j$ satisfies $|\textbf{o}_{i} \cdot \textbf{u}_j|=t_i$ for the index $i$ corresponding to the facet containing $\textbf{u}_j$). Heuristically speaking, we will show that there are only few achievable $h$-tuples and that therefore this situation can be avoided.

In the setting of Theorem \ref{thm:distinct-distances} (where $m=\lceil(1-\mu)n\rceil$ and $A$ has entries in $\mathbb{Q}(x_1,\dots,x_m)$), let us say that an $h$-tuple $(t_1,\dots,t_h)\in \R^h$ is \emph{achievable} if there exist distinct indices $\varphi(1),\dots,\varphi(d\ell+m+1)\in \{1,\dots,h\}$, real numbers $z_1,\dots,z_m\in \mathbb{R}$ and vectors $\textbf{u}_1,\ldots, \textbf{u}_{d\ell+m+1}\in \R^d$ satisfying the following three conditions:
\begin{compactitem}
    \item[(o)] For every entry $a_{ji}\in \mathbb{Q}(x_1,\dots,x_m)$ of the matrix $A$, the evaluation $a_{ji}(z_1,\dots,z_m)\in \mathbb{R}$ is well-defined.
    \item[(a)] $\textbf{u}_{j}=\sum_{i=1}^{\ell} a_{ji}(z_1,\dots,z_m)\textbf{u}_i$ for $j=1,\dots,d\ell+m+1$, and
    \item[(b)] $|\textbf{o}_{\varphi(j)} \cdot \textbf{u}_j|=t_{\varphi(j)}$ for $j=1,\dots,d\ell+m+1$.
\end{compactitem}
Again, intuitively, an $h$-tuple $(t_1,\dots,t_h)$ is achievable if there is a solution $(\textbf{u}_1,\ldots, \textbf{u}_{d\ell+m+1})$ to the linear equations prescribed by the matrix $A$ with the entries evaluated at $(z_1,\dots,z_m)$, such that the vectors $\textbf{u}_1,\ldots, \textbf{u}_{d\ell+m+1}$ satisfy $|\textbf{o}_{i} \cdot \textbf{u}_j|=t_i$ for distinct choices of indices $i$. This will be relevant in the same way as indicated above in the setting of Theorem \ref{thm:main}.

For an injective function $\varphi: \{1,\dots,d\ell+m+1\}\to \{1,\dots,h\}$, let us say that $(t_1,\dots,t_h)\in \R^h$ is \emph{$\varphi$-achievable} if there are vectors $\textbf{u}_1,\ldots, \textbf{u}_{d\ell+1}\in \R^d$ (and real numbers $z_1,\dots,z_m\in \mathbb{R}$) satisfying the conditions above (more precisely, satisfying the conditions (a) to (b) in the setting of Theorem \ref{thm:main} and satisfying the conditions (o) to (b) in the setting of Theorem \ref{thm:distinct-distances}). Then $(t_1,\dots,t_h)$ is achievable if and only if it is $\varphi$-achievable for some injective function $\varphi: \{1,\dots,d\ell+m+1\}\to \{1,\dots,h\}$.

Intuitively speaking, the following two claims show that in both of our two settings only very few $h$-tuples $(t_1,\dots,t_h)\in \R^h$ are achievable.

\begin{claim}\label{claim:achievable-tuples}
In the setting of Theorem \ref{thm:main}, the set of all achievable $h$-tuples $(t_1,\dots,t_h)\in \R^h$ can be covered by finitely many $(h-1)$-dimensional linear hyperplanes in $\R^h$.
\end{claim}

\begin{cla_proof}
Since there are only finitely many possibilities to choose an injective function $\varphi: \{1,\dots,d\ell+1\}\to \{1,\dots,h\}$, it suffices to show that for every such function $\varphi$ the set of all $\varphi$-achievable $h$-tuples $(t_1,\dots,t_h)\in \R^h$ can be covered by finitely many $(h-1)$-dimensional linear hyperplanes in $\R^h$. So let us fix some injective function $\varphi: \{1,\dots,d\ell+1\}\to \{1,\dots,h\}$.

If $(t_1,\dots,t_h)\in \R^h$ is $\varphi$-achievable, there exist vectors $\textbf{u}_1,\ldots, \textbf{u}_{d\ell+1}\in \R^d$ satisfying conditions (a) and (b) above. By (b), for every $j=1,\dots,d\ell+1$, we have $\textbf{o}_{\varphi(j)} \cdot \textbf{u}_j=t_{\varphi(j)}$ or $\textbf{o}_{\varphi(j)} \cdot \textbf{u}_j=-t_{\varphi(j)}$. Note that there are $2^{d\ell+1}$ possibilities to make these sign choices. Fixing the sign choices, we can express each $t_{\varphi(j)}$ for $j=1,\dots,d\ell+1$  as a linear function of the entries of $\textbf{u}_j$ (with coefficients given by the entries of the fixed vector $\textbf{o}_{\varphi(j)}$). However, by (a) the entries of each vector $\textbf{u}_j$ for $j=1,\dots,d\ell+1$ can be expressed as linear functions of the entries of $\textbf{u}_1,\ldots, \textbf{u}_{\ell}$ (with coefficients given by the entries of the fixed matrix $A=(a_{ji})$). Hence we can express each $t_{\varphi(j)}$ for $j=1,\dots,d\ell+1$  as a linear function of the entries of $\textbf{u}_1,\ldots, \textbf{u}_{\ell}\in \R^d$. These $\ell$ vectors have $d\ell$ entries in total, and so each $t_{\varphi(j)}$ for $j=1,\dots,d\ell+1$ is a linear function of the same $d\ell$ real variables. Since the space of all linear functions in $d\ell$ variables only has dimension $d\ell$, there must be a linear dependence between the $d\ell+1$ linear functions expressing $t_{\varphi(j)}$ for $j=1,\dots,d\ell+1$. In other words, we obtain a linear relationship between the values of $t_{\varphi(j)}$ for $j=1,\dots,d\ell+1$ for each of the $2^{d\ell+1}$ sign choices in condition (b). Each such linear relationship gives rise to an $(h-1)$-dimensional linear hyperplane in $\R^h$. Hence the set of $\varphi$-achievable $h$-tuples $(t_1,\dots,t_h)\in \R^h$ can be covered by $2^{d\ell+1}$ linear hyperplanes in $\R^h$.
\end{cla_proof}

\begin{claim}\label{claim:achievable-tuples-2}
In the setting of Theorem \ref{thm:distinct-distances}, the set of all achievable $h$-tuples $(t_1,\dots,t_h)\in \R^h$ can be covered by the vanishing sets $\{(y_1,\dots,y_h)\in \R^h\midd P(y_1,\dots,y_h)=0\}$ of finitely many non-zero polynomials $P\in \mathbb{R}[y_1,\dots,y_h]$.
\end{claim}

\begin{cla_proof}
Again, there are only finitely many possibilities to choose an injective function $\varphi: \{1,\dots,d\ell+m+1\}\to \{1,\dots,h\}$, and so it suffices to show that for every such injective function the set of all $\varphi$-achievable $h$-tuples $(t_1,\dots,t_h)\in \R^h$ can be covered by the vanishing sets of finitely many non-zero polynomials. So let us fix some injective function $\varphi: \{1,\dots,d\ell+m+1\}\to \{1,\dots,h\}$.

If $(t_1,\dots,t_h)\in \R^h$ is $\varphi$-achievable, there exist vectors $\textbf{u}_1,\ldots, \textbf{u}_{d\ell+1}\in \R^d$ satisfying conditions (o) to (b) above. By (b), for every $j=1,\dots,d\ell+m+1$, we have $\textbf{o}_{\varphi(j)} \cdot \textbf{u}_j=t_{\varphi(j)}$ or $\textbf{o}_{\varphi(j)} \cdot \textbf{u}_j=-t_{\varphi(j)}$. Note that there are $2^{d\ell+m+1}$ possibilities to make these sign choices. Fixing the sign choices, we can express each $t_{\varphi(j)}$ for $j=1,\dots,d\ell+m+1$  as a linear function (with real coefficients) of the entries of $\textbf{u}_j$ (with coefficients given by the entries of the fixed vector $\textbf{o}_{\varphi(j)}$). However, by (a) the entries of each vector $\textbf{u}_j$ for $j=1,\dots,d\ell+m+1$ can be expressed as rational functions of $z_1,\dots,z_m$ and the entries of $\textbf{u}_1,\ldots, \textbf{u}_{\ell}$ (where these rational functions are determined by the entries of the fixed matrix $A=(a_{ji})$). Hence we can express each $t_{\varphi(j)}$ for $j=1,\dots,d\ell+m+1$  as a rational function (with real coefficients) of $z_1,\dots,z_m$ and the entries of $\textbf{u}_1,\ldots, \textbf{u}_{\ell}\in \R^d$. The variables $z_1,\dots,z_m$ and the $d\ell$ entries of the  $\ell$ vectors $\textbf{u}_1,\ldots, \textbf{u}_{\ell}\in \R^d$ together are $d\ell+m$ variables in total, and each $t_{\varphi(j)}$ for $j=1,\dots,d\ell+m+1$ is a rational function (with real coefficients) of these $d\ell+m$ variables (where the coefficients of the rational function are all determined by the fixed vectors $\textbf{o}_{1},\dots,\textbf{o}_{h}$ and the fixed matrix $A\in \mathbb{Q}(x_1,\dots,x_m)^{(d\ell+m+1)\times \ell}$). By Fact \ref{fact-trans-degree} there exists a non-zero real polynomial $P$ in $d\ell+m+1$ variables, such that when plugging the rational functions describing $t_{\varphi(1)},\dots,t_{\varphi(d\ell+m+1)}$ into $P$ the resulting expression is zero. This polynomial $P$ is entirely determined by $\textbf{o}_{1},\dots,\textbf{o}_{h}$ and the $A\in \mathbb{Q}(x_1,\dots,x_m)^{(d\ell+m+1)\times \ell}$ (and the sign choices we made above), and we have $P(t_{\varphi(1)},\dots,t_{\varphi(d\ell+m+1)})=0$. Thus, for each of the $2^{d\ell+m+1}$ sign choices in condition (b), we obtain a non-zero polynomial $P\in \mathbb{R}[y_1,\dots,y_h]$ such that $P(t_1,\dots,t_h)=0$ for every $(t_1,\dots,t_h)\in \R^h$ which is $\varphi$-achievable with these sign choices in condition (b). Hence the set of $\varphi$-achievable $h$-tuples $(t_1,\dots,t_h)\in \R^h$ can be covered by the vanishing sets of $2^{d\ell+m+1}$ non-zero polynomials $P\in \mathbb{R}[y_1,\dots,y_h]$.
\end{cla_proof}

Since $B_1$ is bounded, there exists some $c>0$ such that $||\textbf{b}||_2\le c$ for all $\textbf{b}\in B_1$. For any $(t_1,\dots,t_h)\in \R_{>0}^h$, we can consider the polytope $\{\textbf{x}\in \R^d \midd |\textbf{o}_i\cdot \textbf{x}|\le t_i\text{ for }i=1,\dots,h\}$. The maximum (Euclidean) diameter of the facets of this polytope depends continuously on $(t_1,\dots,t_h)$ (indeed, for each $j=1,\dots,h$, the diameter of the subset of those $\textbf{x}\in \{\textbf{x}\in \R^d \midd |\textbf{o}_i\cdot \textbf{x}|\le t_i\text{ for }i=1,\dots,h\}$ maximizing $\textbf{o}_j\cdot \textbf{x}$ depends continuously on $(t_1,\dots,t_h)$). Note that for $(t_1,\dots,t_h)=(s_1,\dots,s_h)$ this polytope is precisely $B_1$ and so all of its facets have diameter at most $s\cdot \sin (\eta/3)$ (with respect to the Euclidean distance). Hence we can choose some $0<\eps'<\frac{\eps s}{2c}$ such that for every $(t_1,\dots,t_h)\in \R_{>0}^h$ with $s_i\le t_i\le s_i+\eps'$ for $i=1,\dots,h$, all facets of the polytope $\{\textbf{x}\in \R^d \midd |\textbf{o}_i\cdot \textbf{x}|\le t_i\text{ for }i=1,\dots,h\}$ have diameter at most $s\cdot \sin (\eta/2)$ (with respect to the Euclidean distance).

By Claim \ref{claim:achievable-tuples} or Claim \ref{claim:achievable-tuples-2}, respectively, the set of achievable $h$-tuples $(t_1,\dots,t_h)\in \R^h$ can be covered by a finite collection of $(h-1)$-dimensional linear hyperplanes in $\R^h$ or by a finite collection of vanishing sets of non-zero polynomials in $\R^h$. Let $H$ be the union of these hyperplanes or of the vanishing sets of these polynomials. In either case, $H\su \R^d$ is a closed set and $H$ cannot contain the entire box $[s_1,s_1+\eps']\times \dots\times [s_h,s_h+\eps']$. Thus, there exists an $h$-tuple $(t_1,\dots,t_h)\not\in H$ with $s_i\le t_i\le s_i+\eps'$ for $i=1,\dots,h$. 

Since $(t_1,\dots,t_h)\not\in H$ and $H$ is closed, there exists some $0<\delta'<\eps'/2$ such that we also have $(t'_1,\dots,t'_h)\not\in H$ for all $(t'_1,\dots,t'_h)\in \R^h$ satisfying $|t_i'-t_i|\le \delta'$ for $i=1,\dots,h$ (geometrically, $\delta'$ is the radius of some ball in the $\ell_\infty$-norm around $(t_1,\dots,t_h)$ that is disjoint from $H$). This means that there is no achievable $(t'_1,\dots,t'_h)$  with $|t_i'-t_i|\le \delta'$ for $i=1,\dots,h$.

Let us now define
\[B=\{\textbf{x}\in \R^d \midd |\textbf{o}_i\cdot \textbf{x}|\le t_i\text{ for }i=1,\dots,h\}\]
to be the convex $\textbf{0}$-symmetric polytope described by the system of linear inequalities $|\textbf{o}_i\cdot \textbf{x}|\le t_i$ for $i=1,\dots,h$. As $t_i\ge s_i\ge s>0$ for $i=1,\dots,h$, the polytope $B$ contains the Euclidean ball of radius $s$ around the origin, and in particular $B$ contains $\textbf{0}$ in its interior. Furthermore, $B$ is bounded, since it is contained within $\max\{t_1/s_1,\dots,t_h/s_h\}B_1$ (which is bounded). Thus, we have $B\in \B_d$.

Since $s_i\le t_i\le s_i+\eps'$ for $i=1,\dots,h$, by the choice of $\eps'$ we know that all facets of the polytope $B$ have diameter at most $s\cdot \sin (\eta/2)$ (with respect to the Euclidean distance).

We claim that $d_H(B,B_1)\le \eps/2$. First, note that
\[B_1=\{\textbf{x}\in \R^d \midd |\textbf{o}_i\cdot \textbf{x}|\le s_i\text{ for }i=1,\dots,h\}\su \{\textbf{x}\in \R^d \midd |\textbf{o}_i\cdot \textbf{x}|\le t_i\text{ for }i=1,\dots,h\}=B,\]
since $s_i\le t_i$ for $i=1,\dots,t$. Thus, for every $\textbf{b}'\in B_1$ we have $\inf_{\textbf{b}\in B} ||\textbf{b}'-\textbf{b}||_2=0$. Furthermore, for every $\textbf{b}\in B$, we have
\[\left|\textbf{o}_i\cdot \frac{s}{s+\eps'}\textbf{b}\right|=\frac{s}{s+\eps'}\cdot |\textbf{o}_i\cdot \textbf{b}|\le \frac{s}{s+\eps'}\cdot t_i\le \frac{s}{s+\eps'}\cdot (s_i+\eps')\le s_i\]
for $i=1,\dots,h$. Hence, letting $\textbf{b}'=\frac{s}{s+\eps'}\textbf{b}$, we have $\textbf{b}'\in B_1$ and $||\textbf{b}-\textbf{b}'||_2= \frac{\eps'}s\cdot ||\textbf{b}'||_2\le \frac{\eps'c}s< \eps/2$, using that $||\textbf{b}'||_2\le c$ (as $\textbf{b}'\in B_1$) and that $\eps'<\frac{\eps s}{2c}$. This shows that $\inf_{\textbf{b}'\in B_1} ||\textbf{b}'-\textbf{b}||_2< \eps/2$ for all $\textbf{b}\in B$. Thus, we indeed have $d_H(B,B_1)\le \eps/2$.

Now, since $d_H(B,B_1)\le \eps/2$ and $d_H(B_0,B_1)< \eps/2$, we have $d_H(B,B_0)<\eps$ by the triangle inequality. In order to finish the proof of Lemma \ref{lem:A-A-eta-nowhere-dense}, let us now show that there is some $\delta>0$ such that no $B'\in \B_d$ with $d_H(B,B')<\delta$ belongs to $\A_{A,\eta}$.

Let us take $\delta$ such that $0<\delta<\delta'$ and small enough such that $(s\cdot \sin (\eta/2)+2\delta)/(s-\delta)\le \sin \eta$. We claim that then for any two points $\textbf{b}, \textbf{b}'\in \R^d$ with Euclidean distance at most $\delta$ from the same facet of $B$, the angle between the lines $\spn_{\R}(\textbf{b})$ and $\spn_{\R}(\textbf{b}')$ is at most $\eta$. Indeed, recalling that each facet of $B$ has diameter at most $s\cdot \sin (\eta/2)$, by the triangle inequality we have $||\textbf{b}-\textbf{b}'||_2\le 2\delta + s\cdot \sin (\eta/2)$. Hence the sine of the angle between the lines $\spn_{\R}(\textbf{b})$ and $\spn_{\R}(\textbf{b}')$ is at most $||\textbf{b}-\textbf{b}'||_2/||\textbf{b}||_2\le (s\cdot \sin (\eta/2)+2\delta)/(s-\delta)\le \sin \eta$ (noting that $||\textbf{b}||_2\ge s-\delta$, since $B$ contains the Euclidean ball of radius $s$ around the origin and $\textbf{b}$ has distance at most $\delta$ to the boundary of $B$). So this indeed shows that the angle between the lines $\spn_{\R}(\textbf{b})$ and $\spn_{\R}(\textbf{b}')$ is at most $\eta$.

It remains to show that there is no $B'\in \B_d$ with $d_H(B,B')<\delta$
belonging to $\A_{A,\eta}$. Suppose towards a contradiction that there
exists some $B' \in \A_{A,\eta}$ such that $d_H(B,B')<\delta$. By
the definition of $\A_{A,\eta}$, this means that there are vectors
$\textbf{u}_1,\ldots,\textbf{u}_{d\ell+m+1}\in \R^d$ with endpoints
on the boundary of $B'$, such that for all $1\le i<j\le d\ell+m+1$, the angle between the two lines $\spn_{\R}(\textbf{u}_i)$ and $\spn_{\R}(\textbf{u}_j)$ is larger than $\eta$. Furthermore, in the setting of Theorem
\ref{thm:main} (where $m=0$), we have $\textbf{u}_{j}=\sum_{i=1}^{\ell}
a_{ji}\textbf{u}_i$ for $j=1,\dots,d\ell+1$. In the setting of
Theorem \ref{thm:distinct-distances} we also have $z_1,\dots,z_m\in
\mathbb{R}$ such that $a_{ji}(z_1,\dots,z_m)$ is well-defined for all
entries $a_{ji}$ of $A$ and such that $\textbf{u}_{j}=\sum_{i=1}^{\ell}
a_{ji}(z_1,\dots,,z_m)\textbf{u}_i$ for $j=1,\dots,d\ell+m+1$. Note that
this means that condition (a) is satisfied in the setting of Theorem
\ref{thm:main} and conditions (o) and (a) are satisfied in the setting
of Theorem \ref{thm:distinct-distances}.

Interpreting $\textbf{u}_1,\ldots,\textbf{u}_{d\ell+m+1}$ as points in $\R^d$, on the boundary of $B'$, by Lemma \ref{lem:Hausdorff-distance-boundary} there exist points $\textbf{b}_1,\ldots,\textbf{b}_{d\ell+m+1}\in \R^d$ on the boundary of $B$ with $||\textbf{u}_j-\textbf{b}_j||_2<\delta$ for $j=1,\dots,d\ell+m+1$ (recall that $d_H(B',B)<\delta$). We claim that no two of the points $\textbf{b}_1,\ldots,\textbf{b}_{d\ell+m+1}$ can lie on the same facet or on opposite facets of the polytope $B$. Indeed, if two of these points $\textbf{b}_j$ and $\textbf{b}_{j'}$ with $j\neq j'$ were on the same facet, then by the choice of $\delta$ the angle between the lines $\spn_{\R}(\textbf{u}_i)$ and $\spn_{\R}(\textbf{u}_j)$ would be at most $\eta$. Similarly, if $\textbf{b}_j$ and $\textbf{b}_{j'}$ with $j\neq j'$ were on opposite facets, then $\textbf{b}_j$ and $-\textbf{b}_{j'}$ would be on the same facet and so the angle between the lines $\spn_{\R}(\textbf{u}_i)$ and $\spn_{\R}(-\textbf{u}_j)=\spn_{\R}(\textbf{u}_j)$ would also be at most $\eta$. In either case, this is a contradiction to the conditions for the vectors $\textbf{u}_1,\ldots, \textbf{u}_{d\ell+m+1}$. Thus, no two of the points $\textbf{b}_1,\ldots,\textbf{b}_{d\ell+m+1}$ can lie on the same facet or on opposite facets of the polytope $B$.

Since $\textbf{b}_1,\ldots,\textbf{b}_{d\ell+m+1}$ are on the boundary of $B=\{\textbf{x}\in \R^d \midd |\textbf{o}_i\cdot \textbf{x}|\le t_i\text{ for }i=1,\dots,h\}$, for every $j=1,\dots,d\ell+m+1$ we can find some $\varphi(j)\in \{1,\dots,h\}$ such that $|\textbf{o}_{\varphi(j)}\cdot \textbf{b}_j|= t_{\varphi(j)}$. Since no two of the points $\textbf{b}_1,\ldots,\textbf{b}_{d\ell+m+1}$ are on the same facet or on opposite facets of $B$, the indices $\varphi(1),\dots, \varphi(d\ell+m+1)\in \{1,\dots,h\}$ must be distinct, and so they define an injective function $\varphi: \{1,\dots,d\ell+m+1\}\to \{1,\dots,h\}$.

For $j=1,\dots,d\ell+m+1$, let us define $t'_{\varphi(j)}=|\textbf{o}_{\varphi(j)}\cdot \textbf{u}_j|$. Since $||\textbf{u}_j-\textbf{b}_j||_2<\delta<\delta'$ and $||\textbf{o}_{\varphi(j)}||_2=1$, we have $|t'_{\varphi(j)}-t_{\varphi(j)}|\le |\textbf{o}_{\varphi(j)}\cdot \textbf{u}_j-\textbf{o}_{\varphi(j)}\cdot \textbf{b}_j|\le ||\textbf{o}_{\varphi(j)}||_2\cdot ||\textbf{u}_j-\textbf{b}_j||_2<\delta'$ for $j=1,\dots,d\ell+m+1$. For every $i\in \{1,\dots,h\}\setminus \{\varphi(1),\dots, \varphi(d\ell+m+1)\}$, let us define $t'_i=t_i$. Then $(t'_1,\dots,t'_h)\in \R^h$ satisfies $|t'_{i}-t_{i}|<\delta'$ for $i=1,\dots,h$, and so by our choice of $\delta'$, the $h$-tuple $(t'_1,\dots,t'_h)$ is not achievable.

On the other hand, the vectors $\textbf{u}_1,\ldots,\textbf{u}_{d\ell+m+1}\in \R^d$ satisfy $|\textbf{o}_{\varphi(j)}\cdot \textbf{u}_j|=t'_{\varphi(j)}$ for $j=1,\dots,d\ell+m+1$, meaning that condition (b) is satisfied in both the setting of Theorem \ref{thm:main} and the setting of Theorem \ref{thm:distinct-distances}. Since we already saw that (a) is satisfied (and also (o) in the setting of Theorem \ref{thm:distinct-distances}), this means that $(t'_1,\dots,t'_h)$ is $\varphi$-achievable and in particular achievable. This contradiction
shows that indeed there is no $B'\in \B_d$ with $d_H(B,B')<\delta$ belonging to $\A_{A,\eta}$. This finishes the proof of Lemma \ref{lem:A-A-eta-nowhere-dense}, and hence of Theorems \ref{thm:main} and \ref{thm:distinct-distances}.
\end{proof}

\section{Proof of Theorem \ref{thm:hadwiger-nelson}: upper bounding the chromatic number}\label{sec:hadwiger-nelson}
In this section, we prove \Cref{thm:hadwiger-nelson}. As mentioned in the introduction, we will actually prove a more general result giving an upper bound for the chromatic number of the odd distance graph of most $d$-norms, in any dimension $d$. The odd distance graph of a $d$-norm $||.||$ is the graph whose vertices are the points in $\mathbb{R}^d$, where two points are adjacent if and only if the distance between them according to the norm $||.||$ is an odd integer. 

\begin{thm}\label{thm:odd-distance-chromatic}
For any $d \ge 1$, for all $d$-norms $||.||$ besides a meagre set, the odd distance graph of $||.||$ has chromatic number at most $2^d$.
\end{thm}

Since the unit distance graph is a subgraph of the odd distance graph, \Cref{thm:odd-distance-chromatic} implies that for all $2$-norms besides a meagre set, the unit distance graph has chromatic number at most $4$. Combined with the fact that this chromatic number is at least $4$ for any $2$-norm \cite[Theorem 3]{Chilakamarri}, this shows \Cref{thm:hadwiger-nelson}.

\begin{proof}[ of \Cref{thm:odd-distance-chromatic}]
Our goal is to show that every $d$-norm $||.||$ for which the odd distance graph has chromatic number larger than $2^{d}$ has the property that there exist vectors $\textbf{u}_1,\ldots,\textbf{u}_k\in \R^d$ in distinct directions with $||\textbf{u}_1||=\dots=||\textbf{u}_k||=1$ and a subset $I\su \{1,\dots,k\}$, such that we have $\mathbf{u}_j\in \spn_{\mathbb{Q}}(\mathbf{u}_i\midd i\in I)$ for at least $d \cdot |I|+1$ indices $j\in \{1,\dots,k\}$. This would be sufficient, as the arguments in the previous section show that the set of $d$-norms with this property is a meagre set (indeed, by the second half of the proof of \Cref{lem:conatined-in-union-A-A-eta} the unit ball of the norm $||.||$ is contained in the union $\bigcup_{A,\eta} \mathcal{A}_{A,\eta}$ considered in the statement of \Cref{lem:conatined-in-union-A-A-eta}, and this union is a meagre set, as each of the sets $\mathcal{A}_{A,\eta}$ is nowhere dense by \Cref{lem:A-A-eta-nowhere-dense}).

So let $||.||$ be a $d$-norm whose odd distance graph has chromatic number larger than $2^{d}$, and suppose towards a contradiction that it does not satisfy the property above. By the De Bruijn--Erd\H{o}s Theorem, the odd distance graph of $||.||$ must contain a finite subgraph $G$ with chromatic number $\chi(G)>2^d$. Let $\textbf{p}_1,\ldots,\textbf{p}_n \in \mathbb{R}^d$ correspond to the vertices of the subgraph $G$. Furthermore, let $\textbf{u}_1,\ldots,\textbf{u}_k\in \mathbb{R}^d$ be the unit vectors (signed arbitrarily) whose odd multiples give rise to edges of $G$. More precisely, $\textbf{u}_1,\ldots,\textbf{u}_k\in \mathbb{R}^d$ is a list of non-zero vectors in distinct directions (i.e.\ the lines $\spn_{\R}(\textbf{u}_i)$ for $i=1,\dots,k$ are all distinct) with $||\textbf{u}_1||=\dots=||\textbf{u}_k||=1$ such that for every edge $\textbf{p}_x\textbf{p}_y$ of $G$ (with $x,y\in\{1,\dots,n\}$), there exists an index $i\in\{1,\dots,k\}$ and an odd integer $m$ with $\textbf{p}_x-\textbf{p}_y=m\textbf{u}_i$.

Recall that by our assumption about the norm $||.||$, for any subset $I\su \{1,\dots,k\}$ we have $\mathbf{u}_j\in \spn_{\mathbb{Q}}(\mathbf{u}_i\midd i\in I)$ for at most $d \cdot |I|$ indices $j\in \{1,\dots,k\}$. By the Edmonds Matroid Decomposition Theorem \cite{edmonds}, this implies that there is a partition $\{1,\dots,k\}=J_1\cup \dots\cup J_d$ such that for each $\ell=1,\dots,d$ the set vectors $\{\mathbf{u}_i\midd i\in J_\ell\}$ is linearly independent over $\mathbb{Q}$. This gives rise to a partition of the edges of the graph $G$ into subgraphs $G_1,\dots,G_d$, where for any $\ell=1,\dots,d$, an edge $\textbf{p}_x\textbf{p}_y$ of $G$ belongs to the subgraph $G_\ell$ if we have $\textbf{p}_x-\textbf{p}_y=m\textbf{u}_i$ for some index $i\in J_\ell$ and an odd integer $m$.

We claim that each of the subgraphs $G_\ell$ for $\ell=1,\dots,d$ is bipartite. Indeed, suppose that for some $\ell\in \{1,\dots,d\}$ the subgraph $G_\ell$ contains an odd cycle. Walking along the cycle and summing up the corresponding terms $\textbf{p}_x-\textbf{p}_y=m\textbf{u}_i$ for the edges along the cycle gives an integer linear relation between the vectors $\textbf{u}_i$ for $i\in J_\ell$. As the cycle is odd, there must be an index $i\in J_\ell$ such that an odd number of edges on the cycle contribute a term of the form $\textbf{p}_x-\textbf{p}_y=m\textbf{u}_i$ (with an odd integer $m$). For this index $i$, the resulting coefficient of $\textbf{u}_i$ in this integer linear relation is odd, and is therefore in particular non-zero. Hence we obtained a non-trivial integer linear relation between the vectors $\textbf{u}_i$ for $i\in J_\ell$, contradicting the fact that these vectors are linearly independent over $\mathbb{Q}$.

Thus, for $\ell=1,\dots,d$, the graph $G_\ell$ must indeed bipartite, meaning that $\chi (G_i) \le 2$. This implies $\chi(G)=\chi(G_1\cup \ldots \cup G_d) \le \chi(G_1) \cdots \chi(G_d) \le 2^d$, giving the desired final contradiction. 
\end{proof}

\section{Proof of Theorem \ref{t12}: point sets with many unit distances}
\label{S5}
This section is concerned with lower bounds for $U_{\nn}(n)$. Before turning to the proofs, let us introduce a standard piece of notation that we will use throughout the section. Given a collection of sets $S_1,\ldots, S_k \in \R^d$ their \emph{Minkowski sum} $S_1+\dots+S_k$ is defined as the set of all points of the form $\textbf{x}_1+\cdots+\textbf{x}_k$ with $\textbf{x}_i \in S_i$ for $i=1,\dots,k$.

\subsection{Planar norms}
Before proving our lower bound for $U_{\nn}(n)$ in
\Cref{t12} in arbitrary dimension, we begin by showing a slightly stronger lower bound in dimension $2$. We note that, as we will explain in \Cref{S6}, it is possible to obtain a further slight improvement.

\begin{prop}
\label{p11}
For every $2$-norm $\|.\|$, we have \[U_{\|.\|}(n) \geq \left(\frac{1}{\log 3}-o(1)\right)\cdot  n \log n\] for all $n$. In other words, for every $2$-norm $\|.\|$ and every $n$, there exists a set of $n$ points in $\R^2$ such that the number of unit distances according to $\|.\|$ among the $n$ points is at least $\left(\frac{1}{\log 3}-o(1)\right) n \log n$, where the $o(1)$-term tends to zero as $n\to \infty$.
\end{prop}

If the norm $\|.\|$ is not strictly convex (i.e.\ if the boundary of the unit ball of the norm contains a line segment), then we can find $n$ points with at least $\lfloor n^2/4 \rfloor$ unit distances. Indeed, in this case, the corresponding unit distance graph contains the complete bipartite graph $K_{m,m}$ for every $m$. So to prove \Cref{p11}, we may assume that the norm is strictly convex.

The main ingredient in the proof is to show that any power of a triangle is a subgraph of the unit distance graph of any strictly convex norm on $\R^2$. More concretely, for any $k$, we find a collection of $k$ triangles with all sides having unit length according to the given norm, and take the Minkowski sum of the vertex sets of these triangles. This yields $3^k$ points and the corresponding unit distance graph is a $k$-fold product of the triangle $K_3$ (in the usual graph theoretic sense of graph products). To find these triangles, we prove the following simple lemma.

\begin{lem}
\label{l21}
Let $\|.\|$ be a strictly convex $2$-norm. Then for any integer $k$ there exist subsets $S_1, S_2, \ldots ,S_k\su \R^2$ of size $|S_1|=\dots=|S_k|=3$ with $|S_1+\dots+S_k|=3^k$ and such that for each $i=1,\dots,k$ any two distinct points in $S_i$ have unit distance according to $||.||$.
\end{lem}

Note that the condition $|S_1+\dots+S_k|=3^k$ means precisely that all the $3^k$ points of the form $\textbf{x}_1+\textbf{x}_2+ \cdots +\textbf{x}_k$ with $\textbf{x}_i \in S_i$ for $i=1,\dots,k$ are distinct. Furthermore, note that the second condition means that for each $i=1,\dots,k$, the three points in $S_i$ determine a triangle with all three side lengths equal to one according to the norm $||.||$.

\begin{proof}
We construct the sets $S_1, S_2, \ldots $ one by one, maintaining the property that after $S_1, S_2, \ldots ,S_i$ have been determined, we have $|S_1+\dots+S_i|=3^i$.
We need the simple fact (see, for example, \cite[Lemma 2]{Br0}) that for any two distinct points $\textbf{a}_1$ and $\textbf{a}_2$ in a strictly convex $2$-dimensional real normed space $(\R^2,||.||)$ there are at most two points $\mathbf{b}\in \R^2$ satisfying both $||\mathbf{b}-\textbf{a}_1||=1$ and $||\mathbf{b}-\textbf{a}_2||=1$ (i.e.\ having unit distance from both $\textbf{a}_1$ and $\textbf{a}_2$).

Assume $S_1, S_2, \ldots, S_{i-1}$ have already been constructed, and let us now choose $S_i$. For any point $\textbf{x}\in \R^2$ with $||\textbf{x}||=1$ (i.e.\ for any point $\textbf{x}$ on the boundary of the unit ball of the norm $\nn$), let us imagine that we move a point $\textbf{y}$ along the boundary of the unit ball of $\nn$ from $\textbf{x}$ to the antipodal point $-\textbf{x}$. By continuity of the distance $||\textbf{y}-\textbf{x}||$ as $\textbf{y}$ moves, at some point in this process we must have $||\textbf{y}-\textbf{x}||=1$. In other words, for any point $\textbf{x}\in \R^2$ with $||\textbf{x}||=1$, there must be a point $\textbf{y}(\textbf{x})\in \R^2$ with $||\textbf{y}(\textbf{x})||=1$ and $||\textbf{y}(\textbf{x})-\textbf{x}||=1$. We will choose the set $S_{i}$ to consist of $\textbf{0}$, $\textbf{x}$ and $\textbf{y}=\textbf{y}(\textbf{x})$, for an appropriate choice of $\textbf{x}$ with $||\textbf{x}||=1$. Note that then any two distinct points in $S_i$ have unit distance according to $||.||$.

In order to ensure that $|S_1+\dots+S_i|=3^i$, we need to to ensure that the sets $\{\textbf{u}+\textbf{0},\textbf{u}+\textbf{x}, \textbf{u}+\textbf{y}\}$ are pairwise disjoint for all $\textbf{u}\in S_1+S_2+ \ldots +S_{i-1}$. If the sets $\{\textbf{u}+\textbf{0},\textbf{u}+\textbf{x}, \textbf{u}+\textbf{y}\}$ and $\{\textbf{v}+\textbf{0},\textbf{v}+\textbf{x}, \textbf{v}+\textbf{y}\}$ intersect for two distinct elements $\textbf{u}, \textbf{v}\in S_1+S_2+ \ldots +S_{i-1}$, then one of the differences $\textbf{x}-\textbf{0},\textbf{y}-\textbf{0},\textbf{x}-\textbf{y}$ must be of the form $\pm (\textbf{u}-\textbf{v})$. We claim that for any two distinct $\textbf{u}, \textbf{v}\in S_1+S_2+ \ldots +S_{i-1}$, there are at most six choices for $\mathbf{x}$ (with $||\textbf{x}||=1$) such that this happens. Indeed, there are clearly at most two choices for $\mathbf{x}$ with $\textbf{x}=\textbf{x}-\textbf{0}\in \{\pm (\textbf{u}-\textbf{v})\}$. Note that $\textbf{x}$ has unit distance from both $\textbf{y}$ and $\textbf{x}-\textbf{y}$ (since $||\textbf{y}-\textbf{x}||=1$ and $||\textbf{y}||=1$). So if $\textbf{y}=\textbf{y}-\textbf{0}\in \{\pm (\textbf{u}-\textbf{v})\}$ or $\textbf{x}-\textbf{y}\in \{\pm (\textbf{u}-\textbf{v})\}$, then $\textbf{x}$ must have unit distance from $\textbf{u}-\textbf{v}$ or from $-(\textbf{u}-\textbf{v})$. But by the above-mentioned fact, there can be at most two choices for $\textbf{x}$ with $||\textbf{x}||=1$ and $||\textbf{x}-(\textbf{u}-\textbf{v})||=1$ and at most two choices for $\textbf{x}$ with $||\textbf{x}||=1$ and $||\textbf{x}+(\textbf{u}-\textbf{v})||=1$. Thus, there are indeed at most six choices for $\mathbf{x}$ (with $||\textbf{x}||=1$) such that the sets $\{\textbf{u}+\textbf{0},\textbf{u}+\textbf{x}, \textbf{u}+\textbf{y}\}$ and $\{\textbf{v}+\textbf{0},\textbf{v}+\textbf{x}, \textbf{v}+\textbf{y}\}$ intersect.

Since there are only finitely many choices for distinct $\textbf{u}, \textbf{v}\in S_1+S_2+ \ldots +S_{i-1}$, this means that there are only finitely many choices for $\textbf{x}$ for which the desired condition $|S_1+\dots+S_i|=3^i$ fails. Since there are infinitely many points $\mathbf{x}\in \R^2$ with $||\textbf{x}||=1$, we can indeed take $S_{i}=\{\textbf{0}, \textbf{x}, \textbf{y}(\textbf{x})\}$ for some suitably chosen $\mathbf{x}\in \R^2$ (with $||\textbf{x}||=1$) such that $|S_1+\dots+S_i|=3^i$.
\end{proof}

Proposition \ref{p11} is an easy consequence of the last lemma.

\begin{proof}[ of \Cref{p11}]
As discussed previously we may assume that the norm $\|.\|$ is strictly convex. Now for any $k$, we can take sets $S_1,\dots,S_k$ as in \Cref{l21} and consider the Minkowski sum $S=S_1+S_2+ \ldots +S_k$. This gives a set $S\su \R^2$ of size $|S|=3^k$, such that  every point in $S$ has unit distance according to $\|.\|$ from at least $2k$ other points in $S$ (namely, all the points in $S$ obtained by changing exactly one summand in the representation $\mathbf{x}_1+\dots+\mathbf{x}_k$ with $\mathbf{x}_i\in S_i$ for $i=1,\dots,k$). This gives us a set of $3^k$ points in $\R^2$ with $k\cdot 3^k$ unit distances according to $\|.\|$, which in particular shows that the claimed bound holds without the $o(1)$-term when $n$ is a power of three.

To obtain the asymptotic bound for every $n$, let us write $n$ in base three, so $n=a_k\cdot 3^k+a_{k-1}\cdot 3^{k-1}+\ldots +a_0$, where $k=\lfloor \log_3 n \rfloor$ and $a_k,\dots,a_0\in\{0,1,2\}$. For each $i=0,\dots,k$, we now use the construction above to find $a_i$ sets of size $3^i$, translated if necessary to ensure that all the points in all of these sets are distinct. In total we obtain $a_k\cdot 3^k+a_{k-1}\cdot 3^{k-1}+\ldots +a_0=n$ points, and among these $n$ points there are at least
\[a_k\cdot k \cdot 3^k+a_{k-1}\cdot (k-1) \cdot 3^{k-1}+\ldots +a_1 \cdot 1 \cdot 3=kn-3^k\cdot \sum_{i=1}^{k}a_{k-i} \cdot i\cdot 3^{-i}\ge kn-2\cdot 3^k\cdot \sum_{i=1}^{k}i\cdot 3^{-i}\ge kn-\frac32 n\]
unit distances. Since $k-\frac32=\left(\frac{1}{\log 3}-o(1)\right)\log n$, this completes the proof. 
\end{proof}

\subsection{Higher dimension}

In this section, we prove Theorem \ref{t12}. We start by showing that the unit distance graph of any $d$-norm contains a complete bipartite graph with one vertex class of size $d-1$ and the other vertex class of infinite size. To do so, we need the following result, known as the Hurewicz Dimension Lowering Theorem (see \cite[Theorem 3.3.10, p.\ 200]{Eng}). The notion of dimension in this theorem is the usual Lebesgue covering dimension, as described in \cite{Eng}.

\begin{thm}[Hurewicz Dimension Lowering Theorem]
\label{t21}
Let $X$ and $Y$ be metric spaces, and assume that $X$ is compact. Let $f: X \to Y$ be a continuous map, such that all fibers $f^{-1}(y)$ for $y \in Y$ have dimension at most $k$. Then we have $\dim(X) \leq k+\dim(Y)$.
\end{thm}

The following lemma is the key ingredient for our lower bound result in \Cref{t12}.

\begin{lem}
\label{c22}
Let $\nn$ be a $d$-norm, and let $\partial B=\{\textbf{x}\in\R^d \midd ||\textbf{x}||=1\}$ be the boundary of the corresponding unit ball. Furthermore, consider distinct points $\textbf{y}_1, \textbf{y}_2 \ldots ,\textbf{y}_{d-1}\in \partial B$, and let $\eps>0$ be a real number. Then there exist points $\textbf{x}_1, \textbf{x}_2 \ldots ,\textbf{x}_{d-1}\in \R^d$ satisfying $\|\textbf{x}_i-\textbf{y}_i\|_2 \leq \eps$ for $i=1,\dots,d-1$, such that the intersection $\bigcap_{i=1}^{d-1} (\textbf{x}_i+\partial B)$ is infinite.
\end{lem}
\begin{proof}
Define a subset $U$ of $\R^d \times (\R^d)^{d-1}$ by
\[U:=\{(\textbf{x},\textbf{x}_1,\ldots ,\textbf{x}_{d-1})\midd  \textbf{x}_i-\textbf{x}\in \partial B\text{ for }i=1,\dots,d-1\}.\]
Note that $U$ is homeomorphic to $\R^d \times (\partial B)^{d-1}$, with a homeomorpism $u:U \to \R^d \times  (\partial B)^{d-1}$ given by $u(\textbf{x},\textbf{x}_1,\ldots,\textbf{x}_{d-1})=(\textbf{x},\textbf{x}_1-\textbf{x},\ldots,\textbf{x}_{d-1}-\textbf{x})$. Now, let us define\begin{equation*}
U':=\{(\textbf{x},\textbf{x}_1, \ldots ,\textbf{x}_{d-1}) \in U \midd \|\textbf{x}_i-\textbf{y}_i\|_2\leq 
\eps \text{ for }i=1,\dots,d-1\}.
\end{equation*}

Observe that the image of $U'$ under the homeomorpism $u:U \to \R^d \times  (\partial B)^{d-1}$ defined above contains the closed set
\begin{equation}\label{eq:contained}
\{\textbf{x}\in \R^d \midd ||\textbf{x}||_2\le \eps/2\} \times \{\textbf{z}_1\in \partial B \midd ||\textbf{z}_1-\textbf{y}_1||_2\le \eps/2\} \times \dots \times  \{\textbf{z}_{d-1}\in \partial B \midd ||\textbf{z}_{d-1}-\textbf{y}_{d-1}||_2\le \eps/2\},
\end{equation}
Indeed, given $(\textbf{x},\textbf{z}_1,\ldots,\textbf{z}_{d-1})$ in this set, we have $u(\textbf{x},\textbf{x}_1,\ldots,\textbf{x}_{d-1})= (\textbf{x},\textbf{z}_1,\ldots,\textbf{z}_{d-1})$ for $(\textbf{x},\textbf{x}_1,\ldots,\textbf{x}_{d-1})\in U'$ given by $\textbf{x}_i=\textbf{z}_i+\textbf{x}$ for $i=1,\dots,d-1$ (then $\textbf{x}_i-\textbf{x}\in \partial B$ and $\|\textbf{x}_i-\textbf{y}_i\|_2=\|\textbf{x}+\textbf{z}_i-\textbf{y}_i\|_2\le \|\textbf{x}\|_2+\|\textbf{z}_i-\textbf{y}_i\|_2\le \eps$).

For $i=1,\dots,d-1$, we have $\textbf{y}_i \in \partial B$ and hence the set $\{\textbf{z}_i\in \partial B \midd ||\textbf{z}_i-\textbf{y}_i||_2\le \eps/2\}$ contains a closed subset that is homeomorphic to $[0,1]^{d-1}$ (indeed, the map $\textbf{z}\mapsto \textbf{z}/||\textbf{z}||_2$ defines a homeomorphism from this set to a subset of the $(d-1)$-dimensional unit sphere in $\R^d$ containing some spherical cap of positive radius). Thus, the set in \eqref{eq:contained} has a closed subset that is homeomorphic to $[0,1]^{d+(d-1)^2}$. Therefore the dimension of the closed set in \eqref{eq:contained}  is at least $d+(d-1)^2$, and consequently $\dim(U')\ge d+(d-1)^2$.

Furthermore, the set $U'$ is closed and bounded (for boundedness, note that $\|\textbf{x}\|\le \|\textbf{x}_1\|+1$ and $\|\textbf{x}_i-\textbf{y}_i\|_2\leq 
\eps$ for $i=1,\dots,d-1$ for all $(\textbf{x},\textbf{x}_1,\ldots,\textbf{x}_{d-1})\in U'$). Thus, $U'$ is compact. Now, let us  consider the continuous projection map $f: U' \to (\R^d)^{d-1}$  given by $f(\textbf{x},\textbf{x}_1,\textbf{x}_2, \ldots ,\textbf{x}_{d-1})=(\textbf{x}_1,\textbf{x}_2, \ldots ,\textbf{x}_{d-1})$.

Since $\dim ((\R^d)^{d-1})=d(d-1)<\dim(U')$, by the Hurewicz Dimension Lowering Theorem applied to the projection map $f: U' \to (\R^d)^{d-1}$, there exist points $\textbf{x}_1,\textbf{x}_2, \ldots ,\textbf{x}_{d-1}\in \R^d$ such that the fiber $f^{-1}(\textbf{x}_1,\textbf{x}_2, \ldots ,\textbf{x}_{d-1})$ has dimension at least one. Then this fiber must be infinite, so there are infinitely many points $\textbf{x}\in \R^d$ with $(\textbf{x},\textbf{x}_1,\textbf{x}_2, \ldots ,\textbf{x}_{d-1})\in U'$. By the definition of $U'$, this implies $\|\textbf{x}_i-\textbf{y}_i\|_2 \leq \eps$ for $i=1,\dots,d-1$. Furthermore, each of the infinitely many points $\textbf{x}$ with $(\textbf{x},\textbf{x}_1,\textbf{x}_2, \ldots ,\textbf{x}_{d-1})\in U'\su U$ satisfies $\textbf{x}_i-\textbf{x}\in \partial B$ for $i=1,\dots,d-1$, meaning that $\textbf{x}-\textbf{x}_i\in \partial B$ and $\textbf{x}\in \textbf{x}_i+\partial B$. Thus, each of these points $\textbf{x}$ lies in the intersection $\bigcap_{i=1}^{d-1} (\textbf{x}_i+\partial B)$, so this intersection is infinite.
\end{proof}

An easy consequence of the last lemma is the following.
\begin{lem}
\label{l23}
Let $k$ be a positive integer and let $\|. \|$ be a $d$-norm. Then there exist points $\textbf{x}_1,\ldots, \textbf{x}_{d-1} \in \R^d$ with $|\{\textbf{x}_1,2\textbf{x}_1,\dots,k\textbf{x}_1\}+\dots+\{\textbf{x}_{d-1},2\textbf{x}_{d-1},\dots,k\textbf{x}_{d-1}\}|=k^{d-1}$ and an infinite set $S\su \R^d$ such that every point $\textbf{z}\in S$ satisfies $\|\textbf{z}-\textbf{x}_i\|=1$ for $i=1,\dots,d-1$.
\end{lem}
\begin{proof}
Let $\partial B$ be the boundary of the unit ball of the norm $\|.\|$. Let us choose points $\textbf{y}_1, \ldots ,\textbf{y}_{d-1}\in \partial B$ such that $|\{\textbf{y}_1,2\textbf{y}_1,\dots,k\textbf{y}_1\}+\dots+\{\textbf{y}_{d-1},2\textbf{y}_{d-1},\dots,k\textbf{y}_{d-1}\}|=k^{d-1}$ (we can choose the points $\textbf{y}_1,\dots,\textbf{y}_{d-1}$ one at a time, maintaining the condition $|\{\textbf{y}_1,2\textbf{y}_1,\dots,k\textbf{y}_1\}+\dots+\{\textbf{y}_{i},2\textbf{y}_i,\dots,k\textbf{y}_{i}\}|=k^{i}$ for $i=1,\dots,d-1$, by observing that after choosing $\textbf{y}_1,\dots,\textbf{y}_{i-1}$ each potential equality between two different sums of the form $a_1\textbf{y}_1+\ldots+a_{i}\textbf{y}_{i}$ with $a_1,\dots,a_{i}\in \{1,\dots,k\}$ forbids only a single choice for $\textbf{y}_i$). Let $\eps>0$ be sufficiently small, such that any two distinct points in $\{\textbf{y}_1,2\textbf{y}_1,\dots,k\textbf{y}_1\}+\dots+\{\textbf{y}_{d-1},2\textbf{y}_{d-1},\dots,k\textbf{y}_{d-1}\}$ have Euclidean distance larger than $2k(d-1) \eps$. Now, by applying \Cref{c22} with $\textbf{y}_1, \ldots ,\textbf{y}_{d-1}$ and $\eps$, we obtain points $\textbf{x}_1,\textbf{x}_2 \ldots ,\textbf{x}_{d-1}$ with $\|\textbf{x}_i-\textbf{y}_i\|_2 \leq \eps$ for $i=1,\dots,d-1$, such that the intersection $\bigcap_{i=1}^{d-1} (\textbf{x}_i+\partial B)$ is infinite. Now, each point in $\{\textbf{x}_1,2\textbf{x}_1,\dots,k\textbf{x}_1\}+\dots+\{\textbf{x}_{d-1},2\textbf{x}_{d-1},\dots,k\textbf{x}_{d-1}\}$ is of the form $a_1\textbf{x}_1+\ldots+a_{d-1}\textbf{x}_{d-1}$ with $a_1,\dots,a_{d-1}\in \{1,\dots,k\}$ and has distance at most $k(d-1)\eps$ from the corresponding point $a_1\textbf{y}_1+\ldots+a_{d-1}\textbf{y}_{d-1}$. By our choice of $\eps$, this means that the points $a_1\textbf{x}_1+\ldots+a_{d-1}\textbf{x}_{d-1}$ with $a_1,\dots,a_{d-1}\in \{1,\dots,k\}$ must all be distinct and hence $|\{\textbf{x}_1,2\textbf{x}_1,\dots,k\textbf{x}_1\}+\dots+\{\textbf{x}_{d-1},2\textbf{x}_{d-1},\dots,k\textbf{x}_{d-1}\}|=k^{d-1}$. Furthermore, defining $S=\bigcap_{i=1}^{d-1} (\textbf{x}_i+\partial B)$, the set $S$ is infinite and every point $\textbf{z}\in S$ satisfies $\|\textbf{z}-\textbf{x}_i\|=1$ for $i=1,\dots,d-1$.
\end{proof}

We are now ready to prove \Cref{t12}.
\begin{proof}[ of \Cref{t12}]
Throughout the argument, we treat $d$ as a fixed constant and all the asymptotics are with respect to large $n$.  Let us choose $m:=\ceil{\log \left(\frac{n}{\log n}\right)}=o(n)$ and $k= \floor{(n/2^m)^{1/(d-1)}}=\Theta((\log n)^{1/(d-1)})$, so that when writing $n'=k^{d-1}\cdot 2^m$, we have $n \ge n'\ge k^{d-1}/(k+1)^{d-1}\cdot n=(1-o(1))n$. Let $\|.\|$ be a $d$-norm.

By Lemma \ref{l23} there exist points $\textbf{x}_1, \ldots ,\textbf{x}_{d-1}\in\R^d$ with $|\{\textbf{x}_1,2\textbf{x}_1,\dots,k\textbf{x}_1\}+\dots+\{\textbf{x}_{d-1},2\textbf{x}_1,\dots,k\textbf{x}_{d-1}\}|=k^{d-1}$ and an infinite subset $S\subseteq \R^d$ such that $\|\textbf{z}-\textbf{x}_i\|=1$ for all $\mathbf{z}\in S$ and $i=1,\dots,d-1$. Now, let us choose points $\textbf{z}_1,\dots ,\textbf{z}_m\in S$ such that the set
\begin{equation}\label{eq:set-lower-bound-proof}
\{\textbf{x}_1,2\textbf{x}_1,\ldots k\textbf{x}_1\}+ \dots +\{\textbf{x}_{d-1},2\textbf{x}_{d-1},\ldots, k\textbf{x}_{d-1}\}
+ \{0,\textbf{z}_1\}+\dots +\{0,\textbf{z}_{m}\}
\end{equation}
has size $k^{d-1}\cdot 2^{m}=n'$ (this is possible by once again choosing the points $\textbf{z}_1,\dots ,\textbf{z}_m\in S$ one by one, observing that at each step every potential equality between two different sums of the form $a_1\textbf{x}_1+\ldots+a_{d-1}\textbf{x}_{d-1}+b_1\textbf{z}_1+\ldots+b_j\textbf{z}_j$ with $a_1,\dots,a_{d-1}\in \{1,\dots,k\}$ and $b_1,\dots,b_j\in \{0,1\}$ forbids at most one choice for $\mathbf{z}_j$).

Now, note that every vector of the form $\textbf{z}_j-\textbf{x}_i$ with $i\in \{1,\dots,d-1\}$ and $j\in \{1,\dots,m\}$ occurrs as the difference of at least $(k-1)\cdot k^{d-2} \cdot 2^{m-1}$ pairs of points in the set in (\ref{eq:set-lower-bound-proof}). Since $\|\textbf{z}_j-\textbf{x}_i\|=1$ for all $i\in \{1,\dots,d-1\}$ and $j\in \{1,\dots,m\}$, this means that among the $n'$ points in the set in (\ref{eq:set-lower-bound-proof}), there are at least 
\[(d-1)m \cdot (k-1)k^{d-2} \cdot 2^{m-1}=\frac{d-1}{2}\cdot \left(1-\frac 1k\right)\cdot (\log n'-(d-1)\log k) \cdot n' = \frac{d-1-o(1)}{2}\cdot n' \log n'\]
unit distances.  Since $n' \log n' =(1-o(1)) n \log n$ and $n' \le n$, this shows there exists a set of at most $n$ points in $\R^d$ with at least $\frac{d-1-o(1)}{2}\cdot n \log n$ unit distances according to the norm $\nn$ (by adding more points we can form such a set with exactly $n$ points).
\end{proof}

\section{Concluding remarks and open problems}
\label{S6}

We have shown in \Cref{t12} that for every $d$-norm $\|.\|$, we have $U_{\|.\|}(n) \geq \frac{1}{2}(d-1-o(1)) \cdot n \log n$. In other words, $\R^d$ contains $n$ points that determine at least $\frac{1}{2} (d-1-o(1))\cdot n \log n$  unit distances according to $\|.\|$, where the $o(1)$-term tends to zero as $n\to \infty$. This is nearly tight, as we have proved in \Cref{thm:main} that there are $d$-norms in which no set of $n$ points determines more than $\frac{1}{2}d \cdot n \log n$ unit distances.  In fact, almost all $d$-norms satisfy this property. For $d=2$ this settles a problem raised by Brass in \cite{Br0} and by Matou\v{s}ek in \cite{matousek}, shaving a $\log \log n$ factor from Matou\v{s}ek's upper bound. For general dimension $d\ge 2$, this provides an essentially tight estimate, up to a $(1-1/d-o(1))$ constant factor, and settles, in a strong and somewhat surprising form, Problems 4 and 5 in \cite[p.\ 195]{BMP-survey}.

For $d=2$ we have shown in \Cref{p11} that $U_{\|.\|}(n) \geq (\frac{1}{\log 3} -o(1))\cdot n \log n$ for every $2$-norm $\|.\|$, improving upon a well-known and often repeated lower bound of $\frac12 \cdot n \log n$ coming from the embedding of the hypercube described in the introduction.  Our upper bound in \Cref{thm:main} for most $2$-norms $\|.\|$ is $ U_{\|.\|}(n) \leq n \log n$, so there is still a gap between the constant factors in the lower and upper bound. In fact, the constant factor $\frac{1}{\log 3}$ in our new lower bound can be slightly improved.  To do so, note that the construction in the proof of \Cref{p11} is a Minkowski sum of sets $S_i = \{0,\textbf{x}_i,\textbf{y}_i\}$. When choosing these sets, the choices of the points $\textbf{x}_i$ are essentially arbitrary, and it is thus possible to choose them so that for every $i \geq 1$, the points $\textbf{x}_{4i-3}+\textbf{x}_{4i-2}$ and $\textbf{x}_{4i-1}+\textbf{x}_{4i}$ have unit distance according to $\|.\|$. Similarly, an additional (tiny) improvement can be obtained by repeating this argument recursively. As this still leaves a gap between the upper and lower bounds, we omit the detailed computation. It may be interesting to close the gap between our upper and lower bounds for all dimensions, although our bounds in \Cref{thm:main,t12} are already quite close (in particular, the bounds get closer as the dimension increases). It seems that this will require some new ideas.  

The proof of Theorem \ref{thm:main} applies the fact, established in Section \ref{S4}, that for most $d$-norms $\|.\|$ there cannot be too many linear dependencies over the rationals between the unit vectors (more precisely, for a given number of unit vectors, their linear span over the rationals cannot contain too many other unit vectors). It is worth noting, however, that some such linear dependencies always exist. In particular, the triangles constructed in the proof of  Lemma \ref{l21} show that in every strictly convex $2$-norm there are infinitely many triples of unit vectors whose sum is $0$. Similarly, a simple application of the Hurewicz Dimension Lowering Theorem (Theorem \ref{t21}) implies that for every $d$-norm $\|.\|$, there are infinitely many $(d-1)$-tuples $(\textbf{x}_1,\dots,\textbf{x}_{d-1})$ of points of unit norm according to $\|.\|$ such that the $(d-2)$-tuple of differences $(\textbf{x}_2-\textbf{x}_1,\dots,\textbf{x}_{d-1}-\textbf{x}_1)$ is the same for all these $(d-1)$-tuples $(\textbf{x}_1,\dots,\textbf{x}_{d-1})$. In particular, this means that for any $\ell\ge d-1$, we can find a set of $\ell$ unit vectors in $\R^d$ whose linear span over the rationals contains at least $(d-1)\cdot (\ell-d+2)$ unit vectors according to $\|.\|$ (in comparison, the arguments in Section \ref{S4} show that for a typical $d$-norm such a span can contain at most $d\cdot \ell$ unit vectors).

Theorem \ref{thm:distinct-distances} provides nearly tight bounds for the largest possible value of $D_{\|.\|}(n)$ for a $d$-norm $\|.\|$ when $n$ is sufficiently large as a function of $d$. 
As far as we know it may be true that an even stronger bound holds, namely that for every fixed $d$, if $n$ is sufficiently large as a function of $d$, then there is a $d$-norm $\|.\|$ so that $D_{\|.\|}(n) =n-1$. This might even hold for most norms. A proof of this, if true, would require some additional arguments. A careful inspection of the computation in our existing proof shows that it implies that for any $d$ and $n$, we have $D_{\|.\|}(n) \geq n-O(dn^{3/4})$ for most $d$-norms $\|.\|$. 
When $n$ is not large as a function of $d$, then the determination of the largest possible value of $D_{\|.\|}(n)$ among all $d$-norms $\|.\|$ is a difficult problem. In particular, the determination of the largest value of $n$ so that $D_{\|.\|}(n)=1$ for all $d$-norms $\|.\|$ is equivalent to a well-known conjecture of Petty, as we describe next.

The number of vertices in the largest clique that can be embedded in the unit distance graph of $\|.\|$ is called the equilateral number of $\|.\|$, which we denote here by $e_{\|.\|}$ (equilateral numbers have also been considered in other types of spaces, see e.g.\ \cite{elsholtz}). A conjecture of Petty raised in 1971 (\cite{Pe}, 
see also \cite{BMP-survey}) asserts that $e_{\|.\|} \geq d+1$ for every $d$-norm $\|.\|$. Note that an equivalent formulation
of this conjecture is that $D_{\|.\|}(n)=1$ for every $d$-norm $\|.\|$ and every $n \leq d+1$ (or equivalently, for $n=d+1$).

This is known in dimension $d \leq 3$ (\cite{Pe}), and is also known for norms that are sufficiently close to the $d$-dimensional Euclidean norm (\cite{Br}, \cite{De}), or more generally to the $d$-dimensional $\ell_p$-norm $\ell_p^d$ for any $1 <p \leq \infty$ (\cite{SV}). Combining this with a result of \cite{AM} asserting that every $d$-dimensional normed space contains a subspace of dimension $r=e^{\Omega(\sqrt {\log d})}$ which is either close to $\ell_2^r$ or to $\ell_{\infty}^r$, it was proved by Swanepoel and Villa \cite{SV} that the equilateral number $e_{\|.\|}$ of any $d$-norm $\|.\|$  is at least $e^{b \sqrt {\log d}}$  for some absolute constant $b>0$. It is interesting to note that the exact value of the equilateral number $e_{\|.\|}$ is not known even for some simple $d$-norms like $\ell_1^d$, where it is conjectured to be $2d$. See \cite{AP}, \cite{Sw0}, \cite{Sw} for more information.

Petty's problem as well as the constructions in \Cref{S5} suggest the investigation of graphs that appear as subgraphs in the unit distance graph of any $d$-norm for some given $d \ge 2$. Any hypercube graph (of any dimension) is an example of such a graph, and a clique of size $e^{b\sqrt {\log d}}$ is another such example (for the absolute constant $b>0$ in the above-mentioned result of Swanepoel and Villa \cite{SV}). The complete bipartite graph $K_{d-1,m}$, for any $m$, is also an example, as proved in \Cref{l23}. Yet another example that appears as a subgraph of the unit distance graph of any strictly convex $2$-norm is the $k$-th power of a triangle (for any $k\ge 1$), as shown in \Cref{l21}. A characterization of all graphs that are subgraphs of the unit distance graph of any $d$-norm appears to be difficult. 

A related problem is to determine the smallest possible chromatic number of the unit distance graph of a $d$-norm. In dimension $d=2$, the answer is $2$ and in fact we showed in \Cref{thm:hadwiger-nelson} that for most $2$-norms the chromatic number is equal to $4$. By the known results about Petty's conjecture, the chromatic number of the unit distance graph of any $d$-norm is always at least $e^{\Omega(\sqrt {\log d})}$. On the other hand, it is not difficult to show that the chromatic number of the unit distance graph of any $d$-norm is at most exponential in $d$ (see e.g.\ \cite{FK}, \cite{Ku}). \Cref{thm:odd-distance-chromatic} implies an upper bound of $2^d$ for most $d$-norms (which gives a better exponential base than the known bounds for all norms). It would be interesting to determine whether the chromatic number of the unit distance graph of a typical $d$-norm is exponential or sub-exponential in $d$. By our results here these graphs are rather sparse, hence one may suspect that the chromatic number might be smaller than exponential.

Let us note that the book of Brass, Moser and Pach \cite{BMP-survey} serves as a remarkable repository of interesting open problems in discrete geometry, many of which have natural extensions to general normed spaces. At least some of the ideas and tools developed in this paper may be helpful in attacking these questions. 

Many of these questions can naturally be posed for typical norms, which in some sense play an analogous role as random graphs do in extremal graph theory. Similarly, as with random graphs, our results showcase that considering typical norms can provide answers to extremal questions for general norms. They also show that the Euclidean norm is very special in many regards, and exhibits a different behaviour with respect to many of these problems than most other norms. In particular, by the known bounds for $D_{\|.\|_2}(n)$ and $U_{\|.\|_2}(n)$ mentioned in the introduction, our results show that the behaviour for the Euclidean $2$-norm with respect to the unit distance problem and the distinct distances problem is very different from the behaviour for a typical $2$-norm, and this difference only becomes more pronounced in higher dimensions. While this might be natural in view of the symmetry of the Euclidean norm, we find it surprising that in comparison, for a typical $d$-norm $\|.\|$, $U_{\|.\|}(n)$ is so small and $D_{\|.\|}(n)$ is so large.

An intriguing open question is to describe explicitly a $d$-norm $\|.\|$ for which $U_{\|.\|}(n)=O(d n \log n)$ or for which $D_{\|.\|}(n) \geq (1-o(1)) n$ (for large $n$). Note that a formal statement of this question requires a definition of the notion ``explicit'' here, a natural one is a norm $\|.\|$ for which there is an efficient algorithm (in any standard model of computation over the reals) for computing the norm $\|\textbf{x}\|$ of any given vector $\textbf{x} \in \R^d$. Note also that it is not even obvious that there exists an explicit $d$-norm as above.

As another application of our methods, we can resolve an analogue of a classical question of Erd\H{o}s and Moser~\cite{erdos-moser} for typical norms. In 1959, Erd\H{o}s and Moser asked about the maximum possible number of unit distances among $n$ points in (strictly) \emph{convex} position in the Euclidean plane, and conjectured that the answer is $O(n)$. F\"uredi~\cite{furedi-convex} proved an upper bound of the form $O(n \log n)$, and the current best bound is still of this form (Aggarwal \cite{aggarwal} improved F\"uredi's bound by a constant factor). The best-known lower bound of $2n-7$ is due to Edelsbrunner and Hajnal \cite{edelsbrunner-hajnal}.

Our arguments show that the conjecture of Erd\H{o}s and Moser is true for most planar norms. More concretely, they imply that for most norms the maximum number of unit distances among $n$ points in a strictly convex position in the plane is at most $4n$. Indeed, the arguments in Section \ref{S4} imply that for most planar norms, for a list of unit vectors $\mathbf{u}_1,\dots,\mathbf{u}_\ell$, the span $\spn_{\mathbb{Q}}(\mathbf{u}_1,\dots,\mathbf{u}_\ell)$ contains a total of at most $2\ell$ unit vectors in different directions (where we consider two unit vectors to have the same direction if they agree up to sign). For any $n$ points in a strictly convex position in the plane, we can consider the graph corresponding to the unit distances between these points and choose a maximal spanning forest in this graph (i.e. we choose a spanning tree for each component of the graph). Let $\mathbf{u}_1,\dots,\mathbf{u}_\ell$ be the unit vectors corresponding to the edges of this spanning forest, and note that $\ell\le n-1$. Now, any unit vector appearing as a distance between two of the $n$ points can be written as a sum of some of the vectors $\mathbf{u}_1,\dots,\mathbf{u}_\ell$ or their negatives (since the two endpoints of the corresponding edge of the graph can be connected by a path in the spanning forest). Hence any unit vector appearing as a distance between the $n$ points must be in $\spn_{\mathbb{Q}}(\mathbf{u}_1,\dots,\mathbf{u}_\ell)$, and so for a typical norm there can be at most $2\ell\le 2n-2$ different directions of such unit vectors. On the other hand, since the $n$ points are in a strictly convex position, for any unit vector $\mathbf{u}$, there can be at most two pairs of points with distance $\pm \mathbf{u}$. Hence, for a typical planar norm, there can be a total of at most $4n-4$ unit distances among $n$ points in a strictly convex position.

Another interesting open problem we suggest is the possible existence of a zero-one law for typical $d$-norms: Is it true that for every fixed $d\ge 2$ and every fixed graph $H$, exactly one of the following two options holds?
\begin{compactitem}
\item For all $d$-norms $\nn$ besides a meagre set, the unit distance graph of $\nn$ contains $H$ as a subgraph.
\item For all $d$-norms $\nn$ besides a meagre set, the unit distance graph of $\nn$ does not contain $H$ as a subgraph.
\end{compactitem}

Finally, we mention that we have another approach for upper-bounding $U_{\|.\|}(n)$ for typical $d$-norms $\|.\|$ yielding a somewhat weaker upper bound than in \Cref{thm:main}, namely of the form $O(d^2n \log n)$. This approach is more similar to Matou\v{s}ek's proof \cite{matousek} showing $U_{\|.\|}(n)\le O(n\log n\log \log n)$ for most $2$-norms $\|.\|$. In particular, by a probabilistic argument using a careful multiple exposure process, we manage to improve the graph-theoretic statement in  \cite[Proposition 2.1]{matousek}, removing the $\log \log n$ factor in this proposition (which causes the $\log \log n$ factor in Matou\v{s}ek's result). 
While the resulting bound $O(d^2n \log n)$ is weaker (in terms of the $d$-dependence) than the bound in \Cref{thm:main}, we plan to write the proof of this improved graph theoretic lemma in a companion note to this paper, since we believe that the argument might be useful in other settings. In particular, a somewhat stronger variant of this lemma would yield an asymptotic answer to the so-called discrete X-Ray reconstruction problem, see \cite{X-ray} for more details.

\textbf{Remark.} While our paper was under review, Greilhuber, Schildkraut, and Tidor \cite{new-lower-bounds} found constructions improving the lower bound for the function $U_{||.||}(n)$ from \Cref{t12} to $(d/2-o(1))\cdot n\log_2 n$. This shows that our \Cref{thm:main} is asymptotically tight.

\textbf{Acknowledgments.} We thank Karim Adiprasito, Pankaj Agarwal, Xiaoyu He, Bo\textquotesingle az Klartag, Mehtaab Sawhney, J\'ozsef Solymosi, Terence Tao and Or Zamir for useful conversations, and J\'anos Koll\'ar and Assaf Naor for leading us to reference \cite{Eng}. The second author gratefully acknowledges the support of the Institute fo Advanced Study in Princeton and the Oswald Veblen fund. We are also very grateful to the anonymous referees for carefully reading the paper and for their helpful comments and suggestions.

\providecommand{\MR}[1]{}
\providecommand{\MRhref}[2]{%
\href{http://www.ams.org/mathscinet-getitem?mr=#1}{#2}}

\providecommand{\href}[2]{#2}


\end{document}